\documentclass[oneside,reqno]{amsart}

\usepackage[T1]{fontenc}
\usepackage[foot]{amsaddr}
\usepackage{adjustbox}
\usepackage{comment}
\usepackage{amsmath}
\usepackage{amsthm}
\usepackage{amssymb}
\usepackage{adjustbox}
\usepackage{diagbox}
\usepackage{mathtools}
\usepackage{tikz, pgfplots}
\usetikzlibrary{positioning}
\usepackage{rotating}
\usepackage{lscape}
\usepackage{multicol}
\usepackage{multirow}
\usepackage[thinlines]{easytable}
\usepackage{cite}
\usepackage{multicol}

\newtheorem{theorem}{Theorem}
\newtheorem*{remark}{Remark}

\newtheorem{proposition}[theorem]{Proposition}

\theoremstyle{definition}
\newtheorem{definition}{Definition}

\title{Computing the Haar state on ${\mathcal{O}(SL_q(3))}$}

\author{Ting Lu}
\address{Math department, Texas A\&M University, College Station, TX 77843, USA}
\email{ting\_lu@tamu.edu}

\begin{document}
\begin{abstract}
This paper shows that to compute the Haar state on $\mathcal{O}(SL_q(n))$, it suffices to compute the Haar states of a special type of monomials which we define as standard monomials. Then, we provide an algorithm to explicitly compute the Haar states of standard monomials on $\mathcal{O}(SL_q(3))$ with reasonable computational cost. The numerical results on $\mathcal{O}(SL_q(3))$ will be used in the future study of the $q$-deformed Weingarten function. 

\end{abstract}

\maketitle

\section{Introduction}
\noindent The Haar measure on a compact topological group is a well-studied object. In particular, when the group is $U(n)$, the group of $n\times n$ unitary matrices, there is an elegant formula for the integral of matrix coefficients with respect to the Haar measure. This formula is given by so-called Weingarten functions, introduced by Collins in 2003~\cite{collins2003moments}. The current paper will study a $q$-deformation of the Haar measure on the Drinfeld--Jimbo~\cite{drinfeld1986quantum}~\cite{jimbo1985aq} quantum groups $\mathcal{O}(SL_q(n))$ which is dual to $U_q(sl_n)$~\cite{korogodski1998algebras}.

\hfill

\noindent In the context of $\mathcal{O}(SL_q(n))$, the most relevant algebraic structure is that it is a co-semisimple Hopf algebra. From Sweedler~\cite{sweedler1969hopf}, any co-semisimple Hopf algebra has a unique ``Haar state'' up to normalization. In the context here, co-semisimplicity plays the role of compactness: the Lie algebra of a compact Lie group is always a semisimple Lie algebra. $\mathcal{O}(SL_q(n))$ becomes a $^*$--Hopf algebra when setting $x_{i,j}^*=S(x_{j,i})$ where $S$ is the antipode. In this case, we call the $^*$--Hopf algebra $\mathcal{O}(SU_q(n))$. In particular, when $q\rightarrow 1$, the space of coordinate functions $O(SU(n))$ on $SU(n)$ is a co-semisimple Hopf algebra, and its Haar state is simply the integral of a function with respect to Haar measure on $SU(n)$. Thus, Haar states serves as the integral functional on these deformed $\mathcal{O}(SU_q(n))$'s. Since the Haar state on $\mathcal{O}(SU_q(n))$ is defined by extending the Haar state on $\mathcal{O}(SL_q(n))$ by the $^*$ operation, we will focus on the Haar state of $\mathcal{O}(SL_q(n))$ only in this paper.

\hfill

\noindent In the $q$-deformed case, there are no explicit formulas in terms of parameter $q$ for the Haar state $\mathcal{O}(SL_q(n))$ except when $n=2$ (Klimyk and Schm{\"u}dgen~\cite{klimyk2012quantum}). The difficulty when $n>2$ arises from the form of the $q$--determinant. When $n=2$, the $q$-determinant is of the form $ad-qbc=1$, where $a,b,c,d$ are the generators of $\mathcal{O}(SL_q(2))$. Because the $q$-determinant only has two terms, once the Haar state of $bc$ is computed, then so is the Haar state of $ad$. However, this simplification does not work in general because the $q$-determinant generally has $n!$ terms. For other related works on $\mathcal{O}(SL_q(n))$, see Nagy~\cite{nagy1993haar}, Vaksman and Soibelman~\cite{vaksman1990algebra}~\cite{levendorskii1991algebras}.

\hfill

\noindent In this paper, the generator of $\mathcal{O}(SL_q(n))$ is denoted as $x_{ij}$ for $1\le i,j\le n$. 
\begin{definition}
The {\bf{counting matrix}} of a monomial $x\in \mathcal{O}(SL_q(n))$, denoted as $\theta(x)$, is a $n\times n$ matrix with entries $a_{ij}, i,j=1,\dots,n$ where $a_{ij}$ equals the number of appearance of generator $x_{ij}$ in $x$.
\end{definition}
\begin{definition}
    The \textbf{{row sum}} and \textbf{{column sum}} of a $n\times n$ matrix $A=(a_{ij})_{i,j=1}^n$, denoted as $\alpha(A)$ and $\beta(A)$, are vectors in $\mathbb{R}^n$:
\begin{equation*}
\begin{split}
    \alpha(A)=\left( \sum_{j=1}^n a_{ij} \right)_{i=1}^n, \ \ \ \ \ \ 
    \beta(A)=\left( \sum_{i=1}^n a_{ij} \right)_{j=1}^n.
\end{split}
\end{equation*}
Here, we denote $(k)_{i=1}^n$ as a vector whose entries all equal to $k$.
\end{definition}
\begin{definition}
    Let $A$ be a $n\times n$ matrix with non-negative integer entries. Then $A$ is a \textbf{$\boldsymbol{k}$-doubly stochastic matrix}~\cite{stein1970enumeration} if there is a positive integer $k$ such that $\alpha(A)=(k)_{i=1}^n=\beta(A)$. 
\end{definition}
\begin{definition}
    Let $S_n$ be the permutation group on $n$ letters. Monomials in form $\prod_{\sigma_i\in S_n}(x_{\sigma_i})^{m_i}$, where $m_i\in \mathbb{N}_0$ and $x_{\sigma_i}=\prod_{k=1}^nx_{k\sigma_i(k)}$ and $(x_{\sigma_i})^0=1$, are called \textbf{standard monomials}. $m=\sum_{i=1}^{n!}m_i$ is called the \textbf{order} and each $x_{\sigma_i}$ is called a \textbf{segment}.
\end{definition}
\noindent The current paper will prove the following theorem on $\mathcal{O}(SL_q(n))$:
\begin{theorem}
The following are true on $\mathcal{O}(SL_q(n))$:
\begin{itemize}
    \item[a)] Let $x$ be a monomial. Then $h(x)\ne 0$ implies that there exist $k\in \mathbb{N}^+$ such that $\theta(x)$ is a ${k}$-doubly stochastic matrix.
    \item[b)] Every monomial with non-zero Haar state value can be written as a linear combination of standard monomials.
    \item[c)]Let $s_l$, $l\in \mathcal{I}_m$ be the set of standard monomials of order $m$. Then, we can write $(Id\otimes h)\circ\Delta(s_l)$ and $(h\otimes Id)\circ\Delta(s_l)$ as linear combinations of $s_j$'s and the coefficient of each $s_j$ is a linear combination of $h(s_i)$'s.
    \item[d)] Let $l(\tau)$ be the inverse number of $\tau\in S_n$. Then :
    \begin{equation*}
    h(\Pi_{k=1}^nx_{k,\tau(k)})=\frac{(-q)^{l(\tau)}}{\sum_{\sigma\in S_n} (-q)^{2l(\sigma)}}=\frac{(-q)^{l(\tau_i)}}{[n]_{q^2}!},
    \end{equation*}
    where $[n]_{q^2}=\frac{1-q^{2n}}{1-q^2}$ and $[n]_{q^2}!=\prod_{j=1}^n[j]_{q^2}$
    \item[e)] When changing the order of generators in a monomial, the newly generated monomials cannot contain more generator $x_{11}$ and $x_{nn}$ and cannot contain less generator $x_{1n}$ and $x_{n1}$, comparing to the monomial being reordered.
\end{itemize}
\end{theorem}
\noindent For simplicity, the generators of on $\mathcal{O}(SL_q(3))$ are denoted as:
\begin{equation*}
\begin{matrix}
a&b&c\\
d&e&f\\
g&h&k.
\end{matrix}
\end{equation*}
Then, standard monomials of order $m=c_1+c_2+c_3+c_4+c_5+c_6$ are in the form:
$$(aek)^{c_1}(afh)^{c_2}(bdk)^{c_3}(bfg)^{c_4}(cdh)^{c_5}(ceg)^{c_6}.$$
\begin{definition}
    Segments $aek$, $afh$, and $bdk$ are {\bf high-complexity segments}. 
\end{definition}
\begin{definition}
    Segments $bfg$, $cdh$, and $ceg$ are {\bf low-complexity segments}.
\end{definition}
\begin{remark}
    Low complexity segments commute with each other but high complexity segments do not commute with any other segments.
\end{remark}
\noindent The current paper will provide an algorithm for explicitly computing the Haar states of monomials on $\mathcal{O}(SL_q(3))$ with $0<|q|<1$. Explicit expressions in terms of parameter $q$ are provided for a special type of monomials, and explicit expressions for general monomials can be computed given enough computational resources.

\hfill

\noindent Using this $q$--deformed Haar measure, we hope to pursue $q$--deformed Weingarten functions in future work. Examples of $q$--deformed Weingarten functions are provided in Appendix~\ref{apd:e}. For all of these examples, it can be seen directly that when $q\rightarrow 1$, the usual Haar measure is recovered. 

\section{Haar state on $\mathcal{O}(SL_q(n))$}
{\noindent}By Noumi \textit{et al.}~\cite{noumi1993finite}, monomials on $\mathcal{O}(GL_q(n))$ form a basis. As a quotient group of $\mathcal{O}(GL_q(n))$, monomials on $\mathcal{O}(SL_q(n))$ form a basis as well. To define the Haar state on $\mathcal{O}(SL_q(n))$, it suffices to define the Haar state of each monomial.

\subsection{Characterization of monomial $\boldsymbol{x}$ such that $\boldsymbol{h(x)\ne 0}$} \hfill \\

\noindent
Not every monomial has a non-zero Haar state value. In this section, we will give a criterion to determine whether the Haar state of a monomial is zero or not.

\hfill

\noindent Let $D_n$ be the diagonal subgroup of $SL_q(n)$. Recall that the coordinate Hopf algebra $\mathcal{O}(D_n)$ is the commutative algebra $\mathbb{C}[t_1,t_1^{-1},\cdots,t_n,t_n^{-1}]$ of all Laurent polynomials in $n$ indeterminates $t_1,t_2,\dots,t_n$ with comultiplication $\Delta(t_i)=t_i\otimes t_i$ and counit $\varepsilon(t_i)=1$. The surjective homomorphism $\pi_{D_n}:\mathcal{O}(SL_q(n))\mapsto \mathcal{O}(D_n)$ is given by $\pi_{D_n}(x_{ij})=\delta_{ij}t_i$. Since we have $D_q^k=1_{\mathcal{O}(SL_q(n))}$ for all $k\in \mathbb{N}^+$, $\pi_{D_n}$ tells us $1_{\mathcal{O}(D_n)}=\pi_{D_n}(1_{\mathcal{O}(SL_q(n))})=\pi_{D_n}(D_q^k)=(\Pi_{i=1}^nt_i)^k$. 

\hfill

\noindent The right and left action of $\mathcal{O}(SL_q(n))$ on $\mathcal{O}(D_n)$, denoted as $L_{D_n}$ and $R_{D_n}$, is defined as:
\begin{equation*}
    \begin{split}
        L_{D_n}&=(\pi_{D_n}\otimes Id)\circ \Delta,\\
        R_{D_n}&=(Id\otimes \pi_{D_n})\circ \Delta.
    \end{split}
\end{equation*}
Given vector $v=(v_1,v_2,\dots,v_n)\in \mathbb{R}^n$, we write $t^v=\Pi_{i=1}^nt_i^{v_i}$. If $x$ is a monomial, we have:
\begin{equation*}
    \begin{split}
        L_{D_n}(x)&=t^{\alpha(\theta(x))}\otimes x,\\
        R_{D_n}(x)&=x\otimes t^{\beta(\theta(x))}.
    \end{split}
\end{equation*}

\hfill

\noindent The next theorem is a generalization of Klimyk and Schmudgen's observation~\cite{klimyk2012quantum}. It gives the necessary condition such that $h(x)\ne 0$ for $x\in \mathcal{O}(SL_q(n))$:

\hfill

\noindent\textbf{Theorem 1 a)}: {\it Let $x$ be a monomial. Then $h(x)\ne 0$ implies that there exist $k\in \mathbb{N}^+$ such that $\theta(x)$ is a ${k}$-doubly stochastic matrix.}

\hfill

\noindent{\bf Proof}:
Consider $(\pi_{D_n}\otimes h)\circ \Delta(x)$. There are two ways to compute this object:
\begin{equation*}
\begin{split}
    (\pi_{D_n}\otimes h)\circ \Delta(x)&=\pi_{D_n}\circ (Id\otimes h)\circ \Delta(x)=\pi_{D_n}(h(x)\cdot 1_{\mathcal{O}(SL_q(n))})\\
    &=h(x)\cdot 1_{\mathcal{O}(D_n)},\\
    (\pi_{D_n}\otimes h)\circ \Delta(x)&=(id \otimes h)\circ (\pi_{D_n}\otimes id)\circ \Delta(x)=(id \otimes h)\circ L_{D_n}(x)\\
    &=(id \otimes h)(t^{\alpha(\theta(x))}\otimes x)\\
    &=h(x)\cdot t^{\alpha(\theta(x))}.
\end{split}
\end{equation*}
Thus, $h(x)\cdot 1_{\mathcal{O}(D_n)}=h(x)\cdot t^{\alpha(\theta(x))}$. Since $h(x)\ne 0$, we get $1_{\mathcal{O}(D_n)}= t^{\alpha(\theta(x))}$. This means that we can find integer $k_1>0$ such that $t^{\alpha(\theta(x))}=(\Pi_{i=1}^nt_i)^{k_1}$. Thus, $\alpha(\theta(x))=(k_1)_{i=1}^n$.

\noindent Apply the same argument to $(h\otimes \pi_{D_n})\circ \Delta(x)$, we get $1_{\mathcal{O}(D_n)}= t^{\beta(\theta(x))}$. Thus, we can find $k_2>0$ such that $\beta(\theta(x))=(k_2)_{i=1}^n$. But we must have $k_1=k_2$ since
\begin{equation*}
    nk_1=\sum_{i=1}^n\sum_{j=1}^na_{ij}=nk_2.  
\end{equation*} 

\hfill $\blacksquare$

% \noindent{\bf Define} $A_n(m),\ m\in \mathbb{N}^+$ as the set of $n\times n$ doubly stochastic matrices such that the sum of each row and the sum of each column equals $m$. 

% \hfill

% \noindent Main theorem a) says that if $h(x)\ne 0$ then the counting matrix $\theta(x)$ must belong to one of these $A_n(m)$'s. So, we have the following definition:

% \hfill

% \noindent{\bf Define} $B_n(m),\ n\in \mathbb{N}^+$ as the set of monomials whose counting matrix belong to $A_n(m)$.

% \hfill

% \noindent We will say monomials in $B_n(m)$ {\bf have order m}. Lemma 1 implies that all the elements in $\mathcal{O}(SL_q(n))$ with non-zero Haar state value is in the linear span of $\cup_{i=1}^{\infty}B_n(i)$. Thus, we define the {\bf non-zero linear space} $\mathcal{NZ}$ as
% \begin{equation*}
%     \mathcal{NZ}=\left\langle \bigcup_{i=1}^{\infty}B_n(i) \right\rangle.
% \end{equation*}
% Notice that here, we use the fact that $1=D_q$ to express the copy of $\mathbb{C}$ in $\mathcal{NZ}$.

\subsection{The linear subspace spanned by monomials with non-zero Haar states}\label{lin_nonzero}

\hfill

\noindent Let $\mathcal{NZ}$ be the linear subspace spanned by monomials with non-zero Haar states. In this section, we give a criterion to pick a basis for $\mathcal{NZ}$. We write $A_n(m),\ m\in \mathbb{N}^+$ as the set of $n\times n$ $m$-doubly stochastic matrices and $B_n(m)$ as the set of monomials on $\mathcal{O}(SL_q(n))$ whose counting matrices belong to $A_n(m)$.

\hfill

\noindent First, we introduce a total order '$<$' on $A_n(m)$.
For every $C=(c_{ij})_{i,jk=1}^n\in A_n(m)$, we associate a vector 
\begin{equation*}
    \mathcal{V}(C)=(c_{11},c_{12},\dots,c_{1n},c_{21},c_{22},\dots,c_{nn}),
\end{equation*}
and we compare such vectors in lexicographic order. We say matrices $C<D$ if $\mathcal{V}(C)<\mathcal{V}(D)$. With this total order, we have the following observation:

\hfill

\noindent{\it If $x=P\cdot x_{ik}x_{jl}\cdot Q\in B_n(m)$ $(i<j,k<l)$ where $P,Q$ are two monomials and we switch the order of $x_{ik}x_{jl}$ so that:}
\begin{equation*}
    x=y+(q-q^{-1})z,
\end{equation*}
{\it where $y=P\cdot x_{jl}x_{ik}\cdot Q$ and $z=P\cdot x_{il}x_{jk}\cdot Q$, then we have: $\theta(z)\in A_n(m)$ and $\theta(z)<\theta(x)=\theta(y)$.} 

\hfill

\noindent Based on the observation, we get the following lemma:\\
{\bf Lemma 1}: For each $A\in A_n(m)$, we fix monomial $x_A\in B_n(m)$ such that $\theta(x_A)=A$. If $\phi\in B_n(m)$ is a monomial with counting matrix $M$, then we can decompose $\phi$ as:
\begin{equation}
    \phi=c_M\cdot x_M+\sum_{\substack{P<M\\P\in A_n(m)}}c_P\cdot x_P.
    \label{eq:1}
\end{equation}
{\bf Proof}: Since $\phi$ and $x_M$ have the same counting matrix, we can permute the generators in $\phi$ to the same order as in $x_M$. We denote this process as a chain:
\begin{equation*}
    \phi=\phi_0\rightarrow \phi_1\rightarrow\phi_2\rightarrow\cdots\rightarrow \phi_k=x_M,
\end{equation*}
where each $\phi_i$ is a reordering of $\phi$ and we get $\phi_{i+1}$ by switching the order of two adjacent generators in $\phi_i$. From $\phi_i$ to $\phi_{i+1}$, we may get a new term $\varphi_{i+1}$. As discussed before, $\theta(\varphi_{i+1})\in A_n(m)$ and $\theta(\varphi_{i+1})<\theta(\phi_{i+1})=M$. We can permute these newly generated $\varphi_i$'s to their corresponding $x_{\theta(\varphi_{i+1})}$'s, and we may get new terms in this process as well. However, each time we repeat this permuting process to a monomial $y$, the counting matrix of the newly generated monomial is always smaller than $\theta(y)$. Since the counting matrix of the newly generated monomial is always descending, we can finish this permuting process in finite steps. In other words, we will get a chain on which every transposition does not generate new monomials. Then, every monomial appearing in the summation will be in the desired form, and we get Equation~(\ref{eq:1}).
\null \hfill $\blacksquare$

\hfill

\noindent Lemma 1 provides a criterion for picking a basis for $\mathcal{NZ}$. Let $S_n(m)=\{x_M, M\in A_n(m)\}$. Then, we can write
\begin{equation*}
    \mathcal{NZ}=\left\langle \bigcup_{i=1}^{\infty}S_n(i) \right\rangle.
\end{equation*}

\hfill

\noindent\textbf{Theorem 1 b)}: \textit{Every monomial with non-zero Haar state value can be written as a linear combination of standard monomials.}

\hfill

 \noindent\textbf{Proof}: By the {\bf Birkhoff-Von Neumann Theorem}~\cite{birkhoff1946three}~\cite{von1953certain}, every $M\in A_n(m)$ can be decomposed into $M=m_1\sigma_1+m_2\sigma_2+\cdots +m_{n!}\sigma_{n!}$, where $\sigma_i$'s are matrix in $A_n(1)$ and $m_i$'s are non-negative integers whose sum is $m$. Notice that each matrix $\sigma_i$ can be identified with a permutation on $n$ letters. We denote the corresponding permutation as $\sigma_i$ as well. Then, the counting matrix of the monomial $\prod_{\sigma_i\in S_n}(x_{\sigma_i})^{m_i}$ is $M$. This implies that for every $M\in A_n(m)$, we can choose $x_M$ in form $\prod_{\sigma_i\in S_n}(x_{\sigma_i})^{m_i}$. Combining Lemma 1, the statement in Main Theorem b) is clear. 
\null \hfill $\blacksquare$

\hfill

\noindent Notice that the set of all standard monomials contains a basis of $\mathcal{NZ}$, but the set itself is not a basis of $\mathcal{NZ}$. The reason is that the different standard monomials could have the same counting matrix. In this case, standard monomials with the same counting matrix are linearly dependent(see Appendix~\ref{apd:c}, Equation (\ref{apeq:7}) and (\ref{apeq:8})). To find a basis of $\mathcal{NZ}$, for each $M\in A(m)$, we have to preserve only one standard monomial corresponding to $M$ and filter out 'unnecessary' standard monomials.

\subsection{Comultiplication of standard monomials}

\hfill \\

\noindent Once we solve the Haar state of standard monomials of each order, we can find the Haar state for other monomials according to their linear combination. We will use the defining relation $((id\otimes h)\circ \Delta)(x)=h(x)\cdot 1=((h\otimes id)\circ \Delta)(x)$ to solve the Haar state of every standard basis. We start with the investigation of the comultiplication of a monomial.

\hfill

\noindent {\bf Lemma 2}: Let $x$ be a monomial and we write:
\begin{equation*}
    \Delta(x)=\sum_{i\in I} z_i\otimes y_i,
\end{equation*}
with $I$ an index set and $y_i,z_i$ monomials. Then we have the following equations:
\begin{equation*}
    \alpha(\theta(x))=\alpha(\theta(z_i))\ \ \ \beta(\theta(x))=\beta(\theta(y_i))\ \ \ \beta(\theta(z_i))=\alpha(\theta(y_i)).
\end{equation*}

\hfill

\noindent {\bf Proof}: Recall that $\Delta(x_{ij})=\sum_{k=1}^nx_{ik}\otimes x_{kj}$ and $\Delta$ is a morphism of algebra. If $x=\Pi_{l=1}^p x_{i_lj_l}$, then
\begin{equation*}
    \Delta(x)=\Delta(\Pi_{l=1}^p x_{i_l,j_l})=\Pi_{l=1}^p\Delta(x_{i_l,j_l})=\prod_{l=1}^p\left(\sum_{k=1}^nx_{i_l,k}\otimes x_{k,j_l}\right).
\end{equation*}
For each $z_i$ the $l$-th generator is in the same row as the $l$-th generator in $x$, and for each $y_i$ the $l$-th generator is in the same column as the $l$-th generator in $x$. The column index of the $l$-th generator in $z_i$ is the same as the row index of the $l$-th generator in $y_i$. Thus, the row sum of $x$ equals the row sum of $z_i$; the column sum of $x$ equals the column sum of $y_i$, and the column sum of $z_i$ equals the row sum of $y_i$. 
\null \hfill $\blacksquare$

\hfill

\noindent With Lemma 2, we have the following result:

\noindent{\bf Lemma 3}: If $\theta(x)\in A_n(m)$ then $h(y_i)\ne 0$ (or $h(z_i)\ne 0$) if and only if $\theta(y_i)\in A_n(m)$ ( or $\theta(z_i)\in A_n(m)$). Moreover, $\theta(y_i)\in A_n(m)$ if and only if $\theta(z_i)\in A_n(m)$.

\hfill

\noindent {\bf Proof}: Use Theorem 1 a) and Lemma 2.
\null \hfill $\blacksquare$

\hfill

\noindent\textbf{Theorem 1 c)}: \textit{Let $\{s_l\}_{l\in \mathcal{I}_m}$,  be the set of standard monomials of order $m$. Then, we can write $(Id\otimes h)\circ\Delta(s_l)$ and $(h\otimes Id)\circ\Delta(s_l)$ as linear combinations of $s_j$'s and the coefficient of each $s_j$ is a linear combination of $h(s_i)$'s.}

\hfill

\noindent\textbf{Proof}: If $s_l\in B_n(m)$ is a standard basis, Lemma 3 implies that
\begin{equation}
    (id\otimes h)\circ\Delta(s_l)=\sum_{\substack{y\in B_n(m)\\ z\in B_n(m)}} h(y)\cdot z=\sum_i h(y_i)\cdot z_i
    \label{ep:4}
\end{equation}
Then, by Lemma 1, we can decompose each $y_i$ and $z_i$ as:
\begin{equation}
    y_i=\sum_{j=1}^{k} d_j^{y_i}\cdot s_j,
    \label{ep:1}
\end{equation}
\begin{equation}
    z_i=\sum_{j=1}^{k} d_j^{z_i}\cdot s_j.
    \label{ep:2}
\end{equation}
Substitute Equation~(\ref{ep:1}) and Equation~(\ref{ep:2}) into Equation~(\ref{ep:4}), we get:
\begin{equation}
    (id\otimes h)\circ\Delta(s_l)=\sum_{j=1}^k\left(\sum_{i=1}^kc_{ij}h(s_i)\right)\cdot s_j.
    \label{eq:2}
\end{equation}
\null \hfill $\blacksquare$
\begin{remark}
Here, $\{s_i\}_{i=1}^k\subset\{s_l\}_{l\in \mathcal{I}_m}$ is a basis of standard monomials of order $m$.
\end{remark}
\noindent In Equation~(\ref{ep:4}), we call $y_i$ the {\bf relation component} and call $z_i$ the {\bf comparing component}. We will say $\boldsymbol{z_i}$ {\bf (or $\boldsymbol{y_i}$) contains} $\boldsymbol{s_j}$ if $d_j^{z_i}\ne 0$ (or $d_j^{y_i}\ne 0$).

\hfill

 \noindent Since we can identify $1$ with $D_q^m$, we get $(id\otimes h)\circ\Delta(s_l)=h(s_l)\cdot D_q^m$. Notice that we can decompose $D_q^m$ as a linear combination of standard monomials of order $m$. Thus, by comparing the coefficient of the same standard monomial on both sides of $(id\otimes h)\circ\Delta(s_l)=h(s_l)\cdot D_q^m$, we can find a linear relation consisting of the Haar states of standard monomials of order $m$ (See section 2.5 for more detail). We call such a linear relation {\bf linear relation of order} $\boldsymbol{m}$. We call a linear system consisting of linear relations of order $m$ a {\bf system of order} $\boldsymbol{m}$.

\subsection{System of order 1}

\hfill \\

\noindent In this section, we will prove Theorem 1 d). The standard basis for $B_n(1)$ is in the form of $x_{\tau_i}=\Pi_{k=1}^nx_{k,\tau_i(k)}$ where $\tau_i$ is a permutation on $n$ letters. We have:
\begin{equation*}
    \Delta(x_{\tau_i})=\Delta(\Pi_{k=1}^nx_{k,\tau_i(k)})=\displaystyle{\prod_{k=1}^n\left(\sum_{p=1}^nx_{k,p}\otimes x_{p,\tau_i(k)}\right)}.
\end{equation*}
By Lemma 3, after apply $(id\otimes h)$ to $\Delta(x_{\tau_i})$, we get:
\begin{equation}
    \begin{split}
        (id\otimes h)\circ\Delta(x_{\tau_i})=\sum_{\sigma_j\in S_n}h( \Pi_{k=1}^nx_{\sigma_j(k),\tau_i(k)})\cdot\Pi_{k=1}^nx_{k,\sigma_j(k)}. 
    \end{split}
\end{equation}
On the other hand, recall that
\begin{equation}
    1=D_q=\sum_{\sigma_j\in S_n} (-q)^{l(\sigma_j)}\prod_{k=1}^nx_{k,\sigma_j(k)},
    \label{eq:dp}
\end{equation}
where $l(\sigma_j)$ is the inverse number of $\sigma_j$.

\hfill

\noindent Thus, using $(id\otimes h)\circ\Delta(x_{\tau_i})=h(x_{\tau_i})\cdot 1$ and comparing the coefficients of each standard basis, we get for every $\sigma_j \in S_n$:
\begin{equation}
    h( \Pi_{k=1}^nx_{\sigma_j(k),\tau_i(k)})=(-q)^{l(\sigma_j)}h(x_{\tau_i}).
    \label{eq:4}
\end{equation}
In general, $\Pi_{k=1}^nx_{\sigma_j(k),\tau_i(k)}$ is not a standard monomial. However, if we choose $\sigma_j$ such that $\sigma_j(k)=n+1-\tau_i(k)$, then every generator in $\Pi_{k=1}^nx_{\sigma_j(k),\tau_i(k)}$ commutes with each other and $\Pi_{k=1}^nx_{\sigma_j(k),\tau_i(k)}=\Pi_{k=1}^nx_{k,n+1-k}$. Moreover, $l(\sigma_j)=\frac{n(n-1)}{2}-l(\tau_i)$. Thus, from Equation~(\ref{eq:4}) we get:
\begin{equation}
    h(\Pi_{k=1}^nx_{k,n+1-k})=(-q)^{\frac{n(n-1)}{2}-l(\tau_i)}h(\Pi_{k=1}^nx_{k,\tau_i(k)}).
    \label{eq:5}
\end{equation}
Therefore, using Equation~(\ref{eq:5}) and Equation~(\ref{eq:dp}) we get:
\begin{equation}
\begin{split}
    1&=h(1)=\sum_{\sigma_j\in S_n} (-q)^{l(\sigma_j)}h(\Pi_{k=1}^nx_{k,\sigma_j(k)})\\
    &=\left(\sum_{\sigma_j\in S_n} (-q)^{2l(\sigma_j)-\frac{n(n-1)}{2}}\right)h(\Pi_{k=1}^nx_{k,n+1-k}),
\end{split}
\end{equation}
which gives
\begin{equation}
    h(\Pi_{k=1}^nx_{k,n+1-k})=\frac{(-q)^{\frac{n(n-1)}{2}}}{\sum_{\sigma_j\in S_n} (-q)^{2l(\sigma_j)}}.
\end{equation}
Then by Equation~(\ref{eq:4}), notice that the inverse number for the $\tau_i$ corresponding to $\Pi_{k=1}^nx_{k,n+1-k}$ is just $\frac{n(n-1)}{2}$, we get for every $\tau_i\in S_n$:
\begin{equation}
    h(\Pi_{k=1}^nx_{k,\tau_i(k)})=\frac{(-q)^{l(\tau_i)}}{\sum_{\sigma_j\in S_n} (-q)^{2l(\sigma_j)}}.
    \label{eq:sys1}
\end{equation}
Now, let $I_n(k)$ be the number of permutations on $n$ letters with $k$ inversions. Then, the denominator of Equation (\ref{eq:sys1}) can be rewritten as:
\begin{equation*}
    \sum_{\sigma_j\in S_n} (-q)^{2l(\sigma_j)}=\sum_{k=0}^{\frac{n(n+1)}{2}}I_n(k)q^{2k}.
\end{equation*}
By Andrews~\cite{andrews1998theory}, the generating function of $I_n(k)$ is
\begin{equation*}
    \sum_{k=0}^{\frac{n(n+1)}{2}}I_n(k)x^{k}=\prod_{j=1}^n\frac{1-x^j}{1-x}. 
\end{equation*}
So the denominator of Equation (\ref{eq:sys1}) can be rewritten as
\begin{equation*}
    \sum_{\sigma_j\in S_n} (-q)^{2l(\sigma_j)}=\prod_{j=1}^n\frac{1-q^{2j}}{1-q^2}=[n]_{q^2}!,
\end{equation*}
and we get:
\begin{equation*}
    h(\Pi_{k=1}^nx_{k,\tau_i(k)})=\frac{(-q)^{l(\tau_i)}}{[n]_{q^2}!}.
\end{equation*}

\subsection{Liner relations of higher order}

\hfill

\hfill

\noindent In this section, let $\{s_l\}_{l=1}^{K_m}$ be a set of linearly independent standard monomials of order $m$. Recall Equation~(\ref{eq:2}):
\begin{equation*}
\begin{split}
    (id\otimes h)\circ\Delta(s_l)=\sum_{j=1}^{K_m}\left(\sum_{i=1}^{K_m}c_{ij}h(s_i)\right)\cdot s_j.
\end{split}
\end{equation*}
We can do the same thing to $h(s_l)\cdot 1=h(s_l)\cdot D_q^m$ and get:
\begin{equation}
    h(s_l)\cdot 1=h(s_l)\cdot D_q^m=\sum_{j=1}^{K_m} b_jh(s_l)\cdot s_j.
    \label{ep:3}
\end{equation}

\hfill

\noindent By comparing the coefficients of standard bases in $(id \otimes h)\circ\Delta(s_l)$ and in $h(s_l)\cdot 1$, we get:
\begin{equation} 
    \sum_{i=1}^{K_m}c_{ij}h(s_i)=b_jh(s_l)
    \label{eq:3}
\end{equation}
for every $1\le j\le K_m$. We will call Equation~(\ref{eq:3}) the \textbf{linear relation derived from equation basis $\boldsymbol{s_l}$ and comparing basis $\boldsymbol{s_j}$}. Each index $1\le l\le K_m$ corresponds to $K_m$ linear relations, so there are $K_m^2$ linear relations. Since there are $K_m$ unknowns, it is possible to construct more than one system of order $m$. Notice that these linear relations all have the zero right-hand side. One way to get a linear relation with the non-zero right-hand side is by decomposing $1=h(1)=h(D_q^m)$ into a sum of standard monomials. Although we can construct more than one system of order $m$, not every system is invertible. We will give a more robust approach to compute the Haar state of $\mathcal{O}(SL_q(3))$ later. \\
\\
In the order 1 case, finding Equation (\ref{eq:2}) and (\ref{ep:3}) is an easy task. However, the situation is much more complicate in higher order case. To understand the difficulty to find the two equations in higher order case, we introduce the \textbf{order restriction} for each summand appearing in the comultiplication of a monomial:\\
\\
\textit{Let $x=\prod_{k\in I}x_{i_k,j_k}$ be a monomial and the comultiplication of $x$ be
\begin{equation*}
    \Delta(x)=\prod_{k\in I}\Delta(x_{i_k,j_k})=\prod_{k\in I}\left(\sum_{l_k=1}^nx_{i_k,l_k}\otimes x_{l_k,j_k}\right)=\sum_{i\in I'} z_i\otimes y_i.
\end{equation*}
Then:
\begin{itemize}
    \item[i)] the $k$-th generator of the left component $z_i$ is in the $i_k$-th row
    \item[ii)] the $k$-th generator of the right component $y_i$ is in the $j_k$-th column
    \item[iii)] The column index of the $l$-th generator in $z_i$ equals to the row index of the $l$-th generator in $y_i$.
\end{itemize}}
\noindent The order restriction is a direct consequence of the fact that the comultiplication is an algebra homomorphism. Since each index $l_k$ ranges from $1$ to $n$, every possible combination of $z_i\otimes y_i$ that satisfies the order restriction will appear in the summation of $\Delta(x)$. In higher order case, this means that Equation (\ref{ep:4}) includes not only summand whose left and right components are standard monomials but also summand whose left and right components are reordering of standard monomials satisfying the order restriction. As an example in $\mathcal{O}(SL_q(3))$, if $(x_{1,1}x_{2,3}x_{3,2})(x_{1,2}x_{2,1}x_{3,3})$ is the left component of one of the tensor products in $\Delta\left((x_{1,1}x_{2,2,}x_{3,3})^2\right)$ then all reordering of the left component satisfying property i) of the order restriction are:
\begin{multicols}{2}
\begin{itemize}
    \item[1)] $(x_{1,1}x_{2,3}x_{3,2})(x_{1,2}x_{2,1}x_{3,3})$
    \item[2)] $(x_{1,2}x_{2,3}x_{3,2})(x_{1,1}x_{2,1}x_{3,3})$
    \item[3)] $(x_{1,1}x_{2,1}x_{3,2})(x_{1,2}x_{2,3}x_{3,3})$
    \item[4)] $(x_{1,1}x_{2,3}x_{3,3})(x_{1,2}x_{2,1}x_{3,2})$
    \item[5)] $(x_{1,2}x_{2,1}x_{3,2})(x_{1,1}x_{2,3}x_{3,3})$
    \item[6)] $(x_{1,2}x_{2,3}x_{3,3})(x_{1,1}x_{2,1}x_{3,2})$
    \item[7)] $(x_{1,1}x_{2,1}x_{3,3})(x_{1,2}x_{2,3}x_{3,2})$
    \item[8)] $(x_{1,2}x_{2,1}x_{3,3})(x_{1,1}x_{2,3}x_{3,2})$.
\end{itemize}
\end{multicols}
\noindent Thus, the comultipilcation of a standard monomial of higher order contains not only standard monomials but also variations of standard monomials satisfying the order restriction. This is the major difference between the case of order 1 and higher order cases. To find a linear relation derived from equation basis $s_l$ and comparing basis $s_j$ in higher order case, we have to:
\begin{itemize}
    \item[i)]find all left (or right) component appearing in $\Delta(s_l)$ that contains $s_j$ and compute the corresponding coefficient $d_j^{z}$ in Equation~(\ref{ep:2});
    \item[ii)] find the decomposition of the right (or left) component in $\Delta(s_l)$ corresponding to the left (or right) component in i) and sum all such decomposition together to get a linear relation;
    \item[iii)] decompose every summand containing $s_j$ in Equation~(\ref{ep:3}) to find $b_j$.
\end{itemize}
All 3 steps involve decomposing non-standard monomials into a linear combination of standard monomials and such decomposition is not easy in general. However, there is a simple criterion to determine whether a standard monomial appears in the decomposition of a non-standard monomial or not.

\hfill

\noindent\textbf{Theorem 1 e)}:\textit{When changing the order of generators in a monomial, the newly generated monomials cannot contain more generator $x_{1,1}$ and $x_{n,n}$ and cannot contain less generator $x_{1,n}$ and $x_{n,1}$ comparing to the monomial being reordered.}

\begin{proof}
    When a new monomial is generated, we replace a pair of $x_{i,k}x_{j,l}$ ($i<j,k<l$) by a pair of $x_{i,l}x_{j,k}$ to get the new monomial. Notice that none of $x_{i,k}$ and $x_{j,l}$ can be $x_{1,n}$ or $x_{n,1}$ and none of $x_{i,l}$ and $x_{j,k}$ can be $x_{1,1}$ or $x_{n,n}$. Thus, $x_{1,n}$ and $x_{n,1}$ can never be replaced by other generators and $x_{1,1}$ and $x_{n,n}$ can never be used as the generator to replace other generators. This finishes the proof.
\end{proof}
\begin{remark}
    The decomposition of a monomial $x$ does not contain those standard monomials whose number of generator $x_{11}$ and $x_{nn}$ (or $x_{1n}$ and $x_{n1}$) exceeds (or less than) that of monomial $x$.
\end{remark}
\noindent Notice that every standard monomial in $\mathcal{O}(SL_q(3))$ contains at least one of $x_{11}$, $x_{13}$, $x_{31}$, and $x_{33}$. Thus, Theorem 1 e) will play an important role in the computation of the Haar state on $\mathcal{O}(SL_q(3))$ later.

\section{A monomial basis for $\mathcal{NZ}$ in $\mathcal{O}(SL_q(3))$ }
\label{chp:8}
\noindent
As mentioned in section \ref{lin_nonzero}, the set of standard monomials is not a basis for linear subspace $\mathcal{NZ}$. In this section, we will provide a criterion to pick a monomial basis for $\mathcal{NZ}$ from the set of standard monomials in $\mathcal{O}(SL_q(3))$ and define a monomials basis for $\mathcal{NZ}$ based on the criterion.
\begin{proposition}\label{prop 6}
    Let $M=(m_{ij})_{i,j=1}^3$ be a $3\times 3$ $k$-doubly stochastic matrix. If there exist $1\le i',j'\le 3$ such that $m_{i',j'}=0$, then $M$ is uniquely decomposed into a linear combination of matrices in $A_3(1)$.
\end{proposition}
\begin{proof}
    Index the six matrices in $A_3(1)$ as:
    \begin{multicols}{3}
        \begin{enumerate}
            \item[] \begin{equation*}
                S_1=\begin{bmatrix}
            1&0&0\\
            0&1&0\\
            0&0&1
        \end{bmatrix}
            \end{equation*}
        \item[] \begin{equation*}
            S_2=\begin{bmatrix}
            1&0&0\\
            0&0&1\\
            0&1&0
        \end{bmatrix}
        \end{equation*}
        \item[] \begin{equation*}
            S_3=\begin{bmatrix}
            0&0&1\\
            0&1&0\\
            1&0&0
        \end{bmatrix}
        \end{equation*}
        \item[] \begin{equation*}
        S_4=\begin{bmatrix}
            0&1&0\\
            1&0&0\\
            0&0&1
        \end{bmatrix}
    \end{equation*}
    \item[] \begin{equation*}
        S_5=\begin{bmatrix}
            0&1&0\\
            0&0&1\\
            1&0&0
        \end{bmatrix}
    \end{equation*}
    \item[] \begin{equation*}
        S_6=\begin{bmatrix}
            0&0&1\\
            1&0&0\\
            0&1&0
        \end{bmatrix}
    \end{equation*}
    \end{enumerate}
    \end{multicols}
    \noindent Without loss of generality, assume that $m_{11}=0$. Then, matrix $S_1$ and $S_2$ cannot appear in the decomposition of $M$. The only matrix in $A_3(1)$ whose $(2,2)$-entry is not zero is $S_3$. Hence, the coefficient of $S_3$ in the decomposition of $M$ is $m_{22}$. Similarly, the only matrix in $A_3(1)$ whose $(3,3)$-entry is not zero is $S_4$, so the coefficient of $S_4$ in the decomposition of $M$ is $m_{33}$; the only matrix in $A_3(1)$ whose $(2,3)$-entry is not zero is $S_5$, so the coefficient of $S_5$ in the decomposition of $M$ is $m_{23}$; the only matrix in $A_3(1)$ whose $(3,2)$-entry is not zero is $S_6$, so the coefficient of $S_6$ in the decomposition of $M$ is $m_{32}$. Hence, $M$ is decomposed into:
    $$M=m_{22}\cdot S_3+m_{33}\cdot S_4+m_{23}\cdot S_5+m_{32}\cdot S_6.$$
    The arguments for other cases are identical to $m_{11}=0$
\end{proof}
\noindent Denote 
\begin{equation*}
    F=\begin{bmatrix}
        1&1&1\\
        1&1&1\\
        1&1&1
    \end{bmatrix}
\end{equation*}
Then, every $3\times 3$ $k$-doubly stochastic matrix $M$ can be written as:
$$M=a\cdot F+N$$
where $a=\text{min}\{m_{i,j},1\le i,j\le 3\}$ and $N$ is a $(k-3a)$-doubly stochastic matrix with at least one entry equals to $0$. By Proposition \ref{prop 6}, the decomposition of $N$ is unique. To define a monomial basis consisting of standard monomials, we need to specify the decomposition of matrix $F$. There are two possible choices for $F$: $aekbfgcdh$ and $afhbdkceg$. They satisfy Equation (\ref{apeq:7}) in Appendix \ref{apd:c}. In this paper, we choose the standard monomial corresponding to matrix $F$ as:
$$aekbfgcdh.$$
Thus, if the unique standard monomial corresponding to $N$ is:
$$(aek)^{n_1}(afh)^{n_2}(bdk)^{n_3}(bfg)^{n_4}(cdh)^{n_5}(ceg)^{n_6},$$
the standard monomial corresponding to $M$ is:
$$(aek)^{n_1+a}(afh)^{n_2}(bdk)^{n_3}(bfg)^{n_4+a}(cdh)^{n_5+a}(ceg)^{n_6}.$$
Notice that in the monomial corresponding to $N$, at least one of $n_1$, $n_4$, and $n_5$ has to be zero since $N$ has a zero entry. For the same reason, at least one of $n_2$, $n_3$, and $n_6$ has to be zero. This implies that for every $k$-doubly stochastic matrix $M$, the corresponding standard monomial contains at most two of the three segments $afh$, $bdk$, $ceg$. Hence, we define the \textbf{monomial basis consisting of standard monomials} as:
$$\{(aek)^{m_1}(afh)^{m_2}(bdk)^{m_3}(bfg)^{m_4}(cdh)^{m_5}(ceg)^{m_6}, m_i\in\mathbf{N}_0 \text{ and } m_2\cdot m_3\cdot m_6=0\}$$
\section[Explicit formulas for special\\ standard monomials on $\mathcal{O}(SL_q(3))$]{Explicit formulas for special standard monomials on $\mathcal{O}(SL_q(3))$}
In this section, we will construct a linear systems of order $m$ called the \textbf{Source Matrix of order} $\boldsymbol{m}$ based on the relation $((id\otimes h)\circ \Delta)(x)=h(x)\cdot 1$ and the explicit solution to the Source Matrix is given. Then, we derive the explicit recursive relation between the Haar state of standard monomials in the form of $(cdh)^l(ceg)^{m-l}$. We start with the motivation of the construction of the Source matrix.\\
\subsection{The Source Matrix of order $m$}

\hfill

\hfill

\noindent Recall the 3 difficulties we introduced in section 2.5. To reduce the computation in step i), we prefer to pick a comparing basis $s_j$ such that the number of $z_i$'s in Equation (\ref{ep:4}) whose decomposition contain $s_j$ is as small as possible. To reduce the computation in step ii), we prefer to pick a equation basis $s_l$ such that in Equation (\ref{ep:4}) the decomposition of the $y_j$'s corresponding to the $z_j$'s in step i) is as simple as possible. To reduce the computation in step iii), we prefer to pick a comparing basis $s_j$ such that the number of terms in the expansion of $D_q^m$ whose decomposition contain $s_j$ is as small as possible.\\
\\
According to Theorem 1 e), the decomposition of a monomial does not contain those standard monomials whose number of generator $a$ and $k$ exceed that of the original monomial. Thus, we should pick those standard monomials containing as many generator $a$ and $k$ as possible to be the comparing basis $s_j$ so that only limited number of $z_j$'s in Equation (\ref{ep:4}) contains $s_j$. Theorem 1 e) also tells us that the decomposition of a monomial contains only those standard monomials whose number of generator $c$ and $g$ equals to or exceeds that of the original monomial. Thus, we should pick those equation basis $s_l$ such that in Equation (\ref{ep:4}) the $y_j$'s corresponding to those $z_j$'s which contains $s_j$ contain as many generator $c$ and $g$ as possible so that the decomposition of $y_j$'s contain only a limited number of standard monomials.\\
\\
Based on the analysis, we pick standard monomial $(ceg)^m$ as the equation basis and consider the linear relation derived from comparing basis $(aek)^{m-1}afh$, $(aek)^{m-1}bdk$, $(aek)^{m-2}afhbdk$, $(aek)^{m-1}bfg$, $(aek)^{m-1}cdh$, and $(aek)^{m-1}ceg$, respectively. Notice that these comparing basis contains at least $m-1$ generator $a$ and at least $m-1$ generator $k$. We exclude the comparing basis $(aek)^m$ since the corresponding linear relation is an identity. According to the order restriction, we list all the terms $z_j\otimes y_j$ in $\Delta\left((ceg)^m\right)$ whose $z_j$ contains one of our chosen comparing basis. Notice that these $z_j$'s are variations of our chosen comparing basis under the order restriction.\\
\\
Variations of $(aek)^{m-1}afh$:
\begin{enumerate}
    \item[1)] $(aek)^lafh(aek)^{m-1-l}\otimes\\(ceg)^lchd(ceg)^{m-1-l}$
    \item[2)] $(aek)^lafk(aek)^kaeh(aek)^{m-2-l-k}\otimes\\(ceg)^lchg(ceg)^{k}ced(ceg)^{m-2-l-k}$
    \item[3)] $(aek)^laeh(aek)^kafk(aek)^{m-2-l-k}\otimes\\(ceg)^lced(ceg)^{k}chg(ceg)^{m-2-l-k}$
\end{enumerate}
Variations of $(aek)^{m-1}bdk$:
\begin{enumerate}
    \item[1)] $(aek)^lbdk(aek)^{m-1-l}\otimes\\
    (ceg)^lfbg(ceg)^{m-1-l}$
    \item[2)] $(aek)^lbek(aek)^kadk(aek)^{m-2-l-k}\otimes\\(ceg)^lfeg(ceg)^kcbg(ceg)^{m-2-l-k}$
    \item[3)] $(aek)^ladk(aek)^kbek(aek)^{m-2-l-k}\otimes\\(ceg)^lcbg(ceg)^kfeg(ceg)^{m-2-l-k}$ 
\end{enumerate}
Variations of $(aek)^{m-2}afhbdk$:
\begin{enumerate}
    \item[1)] $(aek)^{i_1}b(eka)^{i_2}dk(aek)^{i_3}af(kae)^{i_4}h(aek)^{m-2-i_1-i_2-i_3-i_4}\otimes\\(ceg)^{i_1}f(egc)^{i_2}bg(ceg)^{i_3}ch(gce)^{i_4}d(ceg)^{m-2-i_1-i_2-i_3-i_4}$
    \item[2)] $(aek)^{i_1}af(kae)^{i_2}h(aek)^{i_3}b(eka)^{i_4}dk(aek)^{m-2-i_1-i_2-i_3-i_4}\otimes\\(ceg)^{i_1}ch(gce)^{i_2}d(ceg)^{i_3}f(egc)^{i_4}bg(ceg)^{m-2-i_1-i_2-i_3-i_4}$
    \item[3)] $(aek)^{i_1}be(kae)^{i_2}h(aek)^{i_3}afk(aek)^{i_4}adk(aek)^{m-3-i_1-i_2-i_3-i_4}\otimes\\(ceg)^{i_1}fe(gce)^{i_2}d(ceg)^{i_3}chg(ceg)^{i_4}cbg(ceg)^{m-3-i_1-i_2-i_3-i_4}$
    \item[4)] $(aek)^{i_1}afk(aek)^{i_2}adk(aek)^{i_3}be(kae)^{i_4}h(aek)^{m-3-i_1-i_2-i_3-i_4}\otimes\\(ceg)^{i_1}chg(ceg)^{i_2}cbg(ceg)^{i_3}fe(gce)^{i_4}d(ceg)^{m-3-i_1-i_2-i_3-i_4}$ 
    \item[5)] $(aek)^{i_1}adk(aek)^{i_2}be(kae)^{i_3}h(aek)^{i_4}afk(aek)^{m-3-i_1-i_2-i_3-i_4}\otimes\\(ceg)^{i_1}cbg(ceg)^{i_2}fe(gce)^{i_3}d(ceg)^{i_4}chg(ceg)^{m-3-i_1-i_2-i_3-i_4}$  
    \item[6)] $(aek)^{i_1}adk(aek)^{i_2}af(kae)^{i_3}h(aek)^{i_4}bek(aek)^{m-3-i_1-i_2-i_3-i_4}\otimes\\(ceg)^{i_1}cbg(ceg)^{i_2}ch(gce)^{i_3}d(ceg)^{i_4}feg(ceg)^{m-3-i_1-i_2-i_3-i_4}$ 
    \item[7)] $(aek)^{i_1}aeh(aek)^{i_2}b(eka)^{i_3}dk(aek)^{i_4}afk(aek)^{m-3-i_1-i_2-i_3-i_4}\otimes\\(ceg)^{i_1}ced(ceg)^{i_2}f(egc)^{i_3}bg(ceg)^{i_4}chg(ceg)^{m-3-i_1-i_2-i_3-i_4}$  
    \item[8)] $(aek)^{i_1}aeh(aek)^{i_2}afk(aek)^{i_3}adk(aek)^{i_4}bek(aek)^{m-4-i_1-i_2-i_3-i_4}\otimes\\(ceg)^{i_1}ced(ceg)^{i_2}chg(ceg)^{i_3}fbg(ceg)^{i_4}feg(ceg)^{m-4-i_1-i_2-i_3-i_4}$  
    \item[9)] $(aek)^{i_1}b(eka)^{i_2}fk(aek)^{i_3}ad(kae)^{i_4}h(aek)^{m-2-i_1-i_2-i_3-i_4}\otimes\\(ceg)^{i_1}f(egc)^{i_2}hg(ceg)^{i_3}cb(gce)^{i_4}d(ceg)^{m-2-i_1-i_2-i_3-i_4}$ 
    \item[10)] $(aek)^{i_1}b(eka)^{i_2}f(kae)^{i_3}h(aek)^{i_4}adk(aek)^{m-2-i_1-i_2-i_3-i_4}\otimes\\(ceg)^{i_1}f(egc)^{i_2}h(gce)^{i_3}d(ceg)^{i_4}cbg(ceg)^{m-2-i_1-i_2-i_3-i_4}$ 
    \item[11)] $(aek)^{i_1}afk(aek)^{i_2}b(eka)^{i_3}d(kae)^{i_4}h(aek)^{m-2-i_1-i_2-i_3-i_4}\otimes\\(ceg)^{i_1}chg(ceg)^{i_2}f(egc)^{i_3}b(gce)^{i_4}d(ceg)^{m-2-i_1-i_2-i_3-i_4}$ 
    \item[12)] $(aek)^{i_1}afk(aek)^{i_2}be(kae)^{i_3}h(aek)^{i_4}adk(aek)^{m-3-i_1-i_2-i_3-i_4}\otimes\\(ceg)^{i_1}chg(ceg)^{i_2}fe(gce)^{i_3}d(ceg)^{i_4}cbg(ceg)^{m-3-i_1-i_2-i_3-i_4}$ 
    \item[13)] $(aek)^{i_1}ad(kae)^{i_2}h(eka)^{i_3}b(eka)^{i_4}fk(aek)^{m-2-i_1-i_2-i_3-i_4}\otimes\\(ceg)^{i_1}cb(gce)^{i_2}d(ceg)^{i_3}f(egc)^{i_4}hg(ceg)^{m-2-i_1-i_2-i_3-i_4}$ 
    \item[14)] $(aek)^{i_1}ad(kae)^{i_2}h(aek)^{i_3}afk(aek)^{i_4}bek(aek)^{m-3-i_1-i_2-i_3-i_4}\otimes\\(ceg)^{i_1}cb(gce)^{i_2}d(ceg)^{i_3}chg(ceg)^{i_4}feg(ceg)^{m-3-i_1-i_2-i_3-i_4}$ 
    \item[15)] $(aek)^{i_1}aeh(aek)^{i_2}adk(aek)^{i_3}b(eka)^{i_4}fk(aek)^{m-3-i_1-i_2-i_3-i_4}\otimes\\(ceg)^{i_1}ced(ceg)^{i_2}cbg(ceg)^{i_3}f(egc)^{i_4}hg(ceg)^{m-3-i_1-i_2-i_3-i_4}$
    \item[16)] $(aek)^{i_1}aeh(aek)^{i_2}adk(aek)^{i_3}afk(aek)^{i_4}bek(aek)^{m-4-i_1-i_2-i_3-i_4}\otimes\\(ceg)^{i_1}ced(ceg)^{i_2}cbg(ceg)^{i_3}chg(ceg)^{i_4}feg(ceg)^{m-4-i_1-i_2-i_3-i_4}$
    \item[17)] $(aek)^{i_1}b(eka)^{i_2}d(kae)^{i_3}h(aek)^{i_4}afk(aek)^{m-2-i_1-i_2-i_3-i_4}\otimes\\(ceg)^{i_1}f(egc)^{i_2}b(gce)^{i_3}d(ceg)^{i_4}chg(ceg)^{m-2-i_1-i_2-i_3-i_4}$
    \item[18)] $(aek)^{i_1}be(kae)^{i_2}h(aek)^{i_3}adk(aek)^{i_4}afk(aek)^{m-3-i_1-i_2-i_3-i_4}\otimes\\(ceg)^{i_1}fe(gce)^{i_2}d(ceg)^{i_3}cbg(ceg)^{i_4}chg(ceg)^{m-3-i_1-i_2-i_3-i_4}$
    \item[19)] $(aek)^{i_1}afk(aek)^{i_2}ad(kae)^{i_3}h(aek)^{i_4}bek(aek)^{m-3-i_1-i_2-i_3-i_4}\otimes\\(ceg)^{i_1}chg(ceg)^{i_2}cb(gce)^{i_3}d(ceg)^{i_4}feg(ceg)^{m-3-i_1-i_2-i_3-i_4}$
    \item[20)] $(aek)^{i_1}af(kae)^{i_2}h(kae)^{i_3}adk(aek)^{i_4}bek(aek)^{m-3-i_1-i_2-i_3-i_4}\otimes\\(ceg)^{i_1}ch(gce)^{i_2}d(ceg)^{i_3}cbg(ceg)^{i_4}feg(ceg)^{m-3-i_1-i_2-i_3-i_4}$  
    \item[21)] $(aek)^{i_1}adk(aek)^{i_2}b(eka)^{i_3}f(kae)^{i_4}h(aek)^{m-2-i_1-i_2-i_3-i_4}\otimes\\(ceg)^{i_1}cbg(ceg)^{i_2}f(egc)^{i_3}h(gce)^{i_4}d(ceg)^{m-2-i_1-i_2-i_3-i_4}$ 
    \item[22)] $(aek)^{i_1}adk(aek)^{i_2}afk(aek)^{i_3}be(kae)^{i_4}h(aek)^{m-3-i_1-i_2-i_3-i_4}\otimes\\(ceg)^{i_1}cbg(ceg)^{i_2}chg(ceg)^{i_3}fe(gce)^{i_4}d(ceg)^{m-3-i_1-i_2-i_3-i_4}$ 
    \item[23)] $(aek)^{i_1}aeh(aek)^{i_2}b(eka)^{i_3}fk(aek)^{i_4}adk(aek)^{m-3-i_1-i_2-i_3-i_4}\otimes\\(ceg)^{i_1}ced(ceg)^{i_2}f(egc)^{i_3}hg(ceg)^{i_4}cbg(ceg)^{m-3-i_1-i_2-i_3-i_4}$ 
    \item[24)] $(aek)^{i_1}aeh(aek)^{i_2}afk(aek)^{i_3}b(eka)^{i_4}dk(aek)^{m-3-i_1-i_2-i_3-i_4}\otimes\\(ceg)^{i_1}ced(ceg)^{i_2}chg(ceg)^{i_3}f(egc)^{i_4}bg(ceg)^{m-3-i_1-i_2-i_3-i_4}$ 
\end{enumerate}
Variations of $(aek)^{m-1}bfg$:
\begin{enumerate}
    \item[1)] $(aek)^{i_1}b(eka)^{i_2}f(kae)^{i_3}g(aek)^{m-1-i_1-i_2-i_3}\otimes\\(ceg)^{i_1}f(egc)^{i_2}h(gce)^{i_3}a(ceg)^{m-1-i_1-i_2-i_3}$ 
    \item[2)] $(aek)^{i_1}afk(aek)^{i_2}be(kae)^{i_3}g(aek)^{m-2-i_1-i_2-i_3}\otimes\\(ceg)^{i_1}chg(ceg)^{i_2}fe(gce)^{i_3}a(ceg)^{m-2-i_1-i_2-i_3}$ 
    \item[3)] $(aek)^{i_1}be(kae)^{i_2}g(aek)^{i_3}afk(aek)^{m-2-i_1-i_2-i_3}\otimes\\(ceg)^{i_1}fe(gce)^{i_2}a(ceg)^{i_3}chg(ceg)^{m-2-i_1-i_2-i_3}$ 
    \item[4)] $(aek)^{i_1}af(kae)^{i_2}g(aek)^{i_3}bek(aek)^{m-2-i_1-i_2-i_3}\otimes\\(ceg)^{i_1}ch(gce)^{i_2}a(ceg)^{i_3}feg(ceg)^{m-2-i_1-i_2-i_3}$ 
    \item[5)] $(aek)^{i_1}aeg(aek)^{i_2}b(eka)^{i_3}fk(aek)^{m-2-i_1-i_2-i_3}\otimes\\(ceg)^{i_1}cea(ceg)^{i_2}f(egc)^{i_3}hg(ceg)^{m-2-i_1-i_2-i_3}$ 
    \item[6)] $(aek)^{i_1}aeg(aek)^{i_2}afk(aek)^{i_3}bek(aek)^{m-3-i_1-i_2-i_3}\otimes\\(ceg)^{i_1}cea(ceg)^{i_2}chg(ceg)^{i_3}feg(ceg)^{m-3-i_1-i_2-i_3}$ 
\end{enumerate}
Variations of $(aek)^{m-1}cdh$:
\begin{enumerate}
    \item[1)] $(aek)^{i_1}c(eka)^{i_2}d(kae)^{i_3}h(aek)^{m-1-i_1-i_2-i_3}\otimes\\(ceg)^{i_1}k(egc)^{i_2}b(gce)^{i_3}d(ceg)^{m-1-i_1-i_2-i_3}$ 
    \item[2)] $(aek)^{i_1}ce(kae)^{i_2}h(aek)^{i_3}adk(aek)^{m-2-i_1-i_2-i_3}\otimes\\(ceg)^{i_1}k(egc)^{i_2}ed(ceg)^{i_3}cbg(ceg)^{m-2-i_1-i_2-i_3}$ 
    \item[3)] $(aek)^{i_1}aeh(aek)^{i_2}c(eka)^{i_3}dk(aek)^{m-2-i_1-i_2-i_3}\otimes\\(ceg)^{i_1}ced(ceg)^{i_2}k(egc)^{i_3}bg(ceg)^{m-2-i_1-i_2-i_3}$ 
    \item[4)] $(aek)^{i_1}adk(aek)^{i_2}ce(kae)^{i_3}h(aek)^{m-2-i_1-i_2-i_3}\otimes\\(ceg)^{i_1}cbg(ceg)^{i_2}k(egc)^{i_3}ed(ceg)^{m-2-i_1-i_2-i_3}$ 
    \item[5)] $(aek)^{i_1}ad(kae)^{i_2}h(aek)^{i_3}cek(aek)^{m-2-i_1-i_2-i_3}\otimes\\(ceg)^{i_1}cb(gce)^{i_2}d(ceg)^{i_3}keg(ceg)^{m-2-i_1-i_2-i_3}$ 
    \item[6)] $(aek)^{i_1}aeh(aek)^{i_2}adk(aek)^{i_3}cek(aek)^{m-3-i_1-i_2-i_3}\otimes\\(ceg)^{i_1}ced(ceg)^{i_2}cbg(ceg)^{i_3}keg(ceg)^{m-3-i_1-i_2-i_3}$ 
\end{enumerate}
Variations of $(aek)^{m-1}ceg$:
\begin{enumerate}
    \item[1)] $(aek)^laeg(aek)^kcek(aek)^{m-2-l-k}\otimes\\(ceg)^lcea(ceg)^kkeg(ceg)^{m-2-l-k}$ 
    \item[2)] $(aek)^lce(kae)^kg(aek)^{m-2-l-k}\otimes\\(ceg)^lke(gce)^ka(ceg)^{m-1-l-k}$ 
\end{enumerate}
The right components corresponding to our chosen comparing basis contain at least $m-1$ generator $c$ and at least $m-1$ generator $g$. Thus, the unknowns in these linear relations are the Haar state of $aek(ceg)^{m-1}$, $afh(ceg)^{m-1}$, $bdk(ceg)^{m-1}$, $bfgcdh(ceg)^{m-2}$, $bfg(ceg)^{m-1}$, $cdh(ceg)^{m-1}$, and $(ceg)^{m}$. Now we get $6$ linear relations containing $7$ unknowns. To get a solvable linear system, we add the quantum determinant relation:
\begin{equation*}
    \begin{split}
        &h\left((ceg)^{m-1}\right)=h\left(D_q(ceg)^{m-1}\right)\\
        =&h\left(aek(ceg)^{m-1}\right)-q\cdot h\left(afh(ceg)^{m-1}\right)-q\cdot h\left(bdk(ceg)^{m-1}\right)\\
        &+q^2\cdot h\left(bfg(ceg)^{m-1}\right)+q^2\cdot h\left(cdh(ceg)^{m-1}\right)-q^3\cdot h\left(ceg(ceg)^{m-1}\right).
    \end{split}
\end{equation*}
The linear system of order $m$ consisting of the 7 equations is called the \textbf{Source Matrix of order} $\boldsymbol{m}$. Besides the quantum determinant relation, the right-hand-sides of all other linear relations are zero. Thus, the linear system is recursive. The Haar state of $(ceg)^{m-1}$ is solved from the Source Matrix of order $m-1$ and then used as the only non-zero right-hand-side term in the Source Matrix of order $m$. The general solution to the Source Matrix of order $m$ is:
\begin{equation*}
    \begin{split}
        h(aek(ceg)^{m-1})&=\frac{(-q)^{3m-2}(q^2-1)^3(q^4-1)(1+q^4-q^2-q^{2m+2})}{q(q^{2m}-1)^2(q^{2m+2}-1)^2(q^{2m+4}-1)}\\
        h(afh(ceg)^{m-1})&=\frac{(-q)^{3m-2}(q^2-1)^4(q^4-1)}{(q^{2m}-1)^2(q^{2m+2}-1)^2(q^{2m+4}-1)}\\
        h(bdk(ceg)^{m-1})&=\frac{(-q)^{3m-2}(q^2-1)^4(q^4-1)}{(q^{2m}-1)^2(q^{2m+2}-1)^2(q^{2m+4}-1)}\\
        h(bfgcdh(ceg)^{m-2})&=\frac{(-q)^{3m-2}(q^2-1)^4(q^4-1)}{(q^{2m}-1)^2(q^{2m+2}-1)^2(q^{2m+4}-1)}\\
        h(bfg(ceg)^{m-1})&=\frac{(-q)^{3m-1}(q^2-1)^3(q^4-1)}{(q^{2m}-1)(q^{2m+2}-1)^2(q^{2m+4}-1)}\\
        h(cdh(ceg)^{m-1})&=\frac{(-q)^{3m-1}(q^2-1)^3(q^4-1)}{(q^{2m}-1)(q^{2m+2}-1)^2(q^{2m+4}-1)}\\
        h((ceg)^{m})&=\frac{(-q)^{3m}(q^2-1)^2(q^4-1)}{(q^{2m+2}-1)^2(q^{2m+4}-1)}.
    \end{split}
\end{equation*}
For the entries of the Source matrix of order $m$, see Table~\ref{tab:source} on the next page.\\
\\

\begin{landscape}
%\begin{TAB}(r,0cm,0cm)[3pt]
\begin{table}[htb!]
    \centering
    \caption{The source matrix with all its entries.}
\resizebox{1.65\textheight}{!}{
\begin{tabular}{c|cccccccc}
%{c|ccccccc}
%\hline
 & $aekceg$ & $afhceg$ & $bdkceg$ & $bfgcdh$ & $bfgceg$ & $cdhceg$ & $cegceg$ & LHS \\ \hline
$D_q$ & $1$ & $-q$ & $-q$ & $0$ & $q^2$ & $q^2$ & $-q^3$ & $h((ceg)^{m-1})$ \\
 $aekceg$&$\frac{q^2(q^{2n}-1)^2}{q^{2n}(q^2-1)^2}$  & $\frac{-q(q^{2n}-1)^2}{q^{2n}(q^2-1)}$ & $\frac{q^3(1-q^{2n})^3}{q^{4n}(q^2-1)^2}$ & $\frac{(q^3-q^{2n+1})(q^{2n}-1)^3}{q^{4n}(q^2-1)^2}$ & $\frac{(q^{2n}-1)^2}{q^{2n}}$ & $\frac{n(q^{2n}-1)^2}{q^{2n}}$ & $\frac{(q^{2n}-1)((n+1)q^4-2nq^2+n)}{q^{2n+1}(q^2-1)}$ & 0 \\
 $aekbfg$& 0 & $\frac{q^2(q^{2n}-1)^3}{q^{2n}(q^2-1)^3}$ & 0 & 0 & $\frac{-q(q^{2n}-1)^2}{q^{2n}(q^2-1)}$ & $\frac{-nq(q^{2n}-1)^2}{q^{2n}(q^2-1)}$ & $\frac{(q^{2n}-1)(n-(n+1)q^2)}{q^{2n}(q^2-1)}$ & 0 \\ 
 $aekcdh$& 0 & 0 & $\frac{q^4(q^{2n}-1)^3}{q^{4n}(q^2-1)^3}$ & $\frac{(q^{2n}-1)^3(q^{2n+2}-q^4)}{q^{4n}(q^2-1)^3}$ & $\frac{-q(q^{2n}-1)^2}{q^{2n}(q^2-1)}$ & $\frac{-nq(q^{2n}-1)^2}{q^{2n}(q^2-1)}$ & $\frac{(q^{2n}-1)(n-(n+1)q^2)}{q^{2n}(q^2-1)}$ & 0 \\
 $afhbdk$& 0 & 0 & 0 & $\frac{(q^{2n}-1)^3(q^{2n+2}-q^4)}{q^{4n}(q^2-1)^4}$ & $c_1$ & $c_2$ & $c_3$ & 0 \\
 $aekbdk$& 0 & 0 & 0 & 0 & $\frac{q^2(q^{2n}-1)^2}{q^{2n}(q^2-1)^2}$ & 0 & $\frac{q(q^{2n}-1)}{q^{2n}(q^2-1)}$ & 0 \\
 $aekafh$& 0 & 0 & 0 & 0 & 0 & $\frac{q^2(q^{2n}-1)^2}{q^{2n}(q^2-1)^2}$ & $\frac{q(q^{2n}-1)}{q^{2n}(q^2-1)}$ & 0 \\ %\hline
\end{tabular}
}
\label{tab:source}
\end{table}

%\end{TAB}
\begin{equation*}
    \begin{split}
        c_1&=\frac{(q^{2n}-1)^2(q^{2n-2}+(n-1)q^{2n}-(n-1)q^{2n+2}-1)}{q^{4n-3}(q^2-1)^3}\\
        c_2&=\frac{(q^{2n}-1)^2(nq^{2n-2}-(n-1)q^{2n}-1)}{q^{4n-3}(q^2-1)^3}\\
        c_3&=\frac{(q^{2n}-1)(q^{2n+4}-n(q^4-1)q^{2n}-q^2)}{q^{4n}(q^2-1)^2}
    \end{split}
\end{equation*}
\end{landscape}

\noindent Through direct computation, we can verify that the Source Matrix of order $2$ and $3$ fit in the general form.  In following subsections we provide the steps to compute the contribution of tensor products in the forms of:
\begin{enumerate}
    \item[1)] $(aek)^lafh(aek)^{m-1-l}\otimes\\(ceg)^lchd(ceg)^{m-1-l}$
    \item[2)] $(aek)^lafk(aek)^kaeh(aek)^{m-2-l-k}\otimes\\(ceg)^lchg(ceg)^{k}ced(ceg)^{m-2-l-k}$
    \item[3)] $(aek)^laeh(aek)^kafk(aek)^{m-2-l-k}\otimes\\(ceg)^lced(ceg)^{k}chg(ceg)^{m-2-l-k}$
\end{enumerate}
to each linear relation in the Source Matrix. The contributions of other tensor products to the Source Matrix are computed in a similar way.

\subsubsection{Decomposition of left components}\label{2.1.1}

\hfill

\hfill

\noindent We start with left components in the form of $(aek)^lafh(aek)^{m-1-l}$. Besides considering the coefficient of $(aek)^{m-1}afh$, the decomposition of the above Left components may contains other standard monomials with $m-1$ generator $a$ and $m-1$ generator $k$. Since there is only $m-1$ generator $k$'s in these left components, we may ignore the new monomials generated by switching $k$ with other generators.
\begin{equation*}
    \begin{split}
        &(aek)^lafh(aek)^{m-1-l}\\
        =&(aek)^{m-1}afh+(q-q^{-1})\sum_{i=1}^{m-1-l}(aek)^{l+i-1}abdk(aek)^{m-1-l-i}fh+\cdots\\
        &-(q-q^{-1})*\sum_{i=1}^{m-1-l}(aek)^la(aek)^{i-1}cdek(aek)^{m-1-l-i}h\\
        &-(q-q^{-1})*\sum_{i=1}^{m-1-l}(aek)^laf(aek)^{i-1}bgek(aek)^{m-1-l-i}
    \end{split}
\end{equation*}
Notice that every term appearing in the decomposition contains exactly $m-1$ generator $a$ and $m-1$ generator $k$ besides $(aek)^{m-1}afh$. Thus, we have:
\begin{equation*}
    \begin{split}
        &(aek)^{l+i-1}abdk(aek)^{m-1-l-i}fh\\
        =&(aek)^{m-2}afhbdk+(q^2-1)*(aek)^{m-1}bfg+(q^2-1)*(aek)^{m-1}cdh\\
        &-(q^2-1)^2/q*(aek)^{m-1}ceg\cdots,
    \end{split}
\end{equation*}
and
\begin{equation*}
    \begin{split}
        &(aek)^la(aek)^{i-1}cdek(aek)^{m-1-l-i}h\\
        =&q^2*(aek)^{m-1}cdh+\cdots,
    \end{split}
\end{equation*}
and
\begin{equation*}
    \begin{split}
        &(aek)^laf(aek)^{i-1}bgek(aek)^{m-1-l-i}\\
        =&q^2*(aek)^{m-1}bfg-(q^3-q)*(aek)^{m-1}ceg+\cdots.
    \end{split}
\end{equation*}
Together, we have:
\begin{equation*}
    \begin{split}
        &(aek)^lafh(aek)^{m-1-l}\\
        =&(aek)^{m-1}afh+(q-q^{-1})(m-1-l)*(aek)^{m-2}afhbdk\\
        &-(q-q^{-1})(m-1-l)*(aek)^{m-1}bfg-(q-q^{-1})(m-1-l)*(aek)^{m-1}cdh\\
        &+(q-q^{-1})^2(m-1-l)*(aek)^{m-1}ceg+\cdots
    \end{split}
\end{equation*}
Then, we consider left components in the form of $(aek)^lafk(aek)^kaeh(aek)^{m-2-l-k}$. Again, we may ignore the new monomials generated from switching generator $k$ with other generators. Thus, we get:
\begin{equation*}
    (aek)^lafk(aek)^kaeh(aek)^{m-2-l-k}=(aek)^laf(aek)^{k+1}h(aek)^{m-2-l-k}+\cdots
\end{equation*}
For $(aek)^laf(aek)^{k+1}h(aek)^{m-2-l-k}$, we have:
\begin{equation*}
    \begin{split}
        &(aek)^laf(aek)^{k+1}h(aek)^{m-2-l-k}\\
        =&(aek)^lafh(aek)^{m-1-l}\\
        &+(q^2-1)q(k+1)*(aek)^lbfg(aek)^{m-2-l}\\
        &-(q^2-1)^2(k+1)*(aek)^lceg(aek)^{m-2-l}+\cdots\\
    \end{split}
\end{equation*}
Then, using the result of the decomposition of $(aek)^lafh(aek)^{m-1-l}$, we get:
\begin{equation*}
    \begin{split}
        &(aek)^lafk(aek)^kaeh(aek)^{m-2-l-k}\\
        =&(aek)^{m-1}afh+(q-q^{-1})(m-1-l)*(aek)^{m-2}afhbdk\\
        &+(q-q^{-1})\left[(k+1)q^2-(m-1-l)\right]*(aek)^{m-1}bfg\\
        &-(q-q^{-1})(m-1-l)*(aek)^{m-1}cdh\\
        &+(q-q^{-1})^2\left[(m-1-l)-(k+1)q^2\right]*(aek)^{m-1}ceg+\cdots
    \end{split}
\end{equation*}
Finally we consider left components in the form of $(aek)^laeh(aek)^kafk(aek)^{m-2-l-k}$. Similar to previous two cases, we may ignore the new monomials generated from switching generator $k$ with other generators. Thus, we get:
\begin{equation*}
    (aek)^laeh(aek)^kafk(aek)^{m-2-l-k}=q^2*(aek)^{l+1}h(aek)^{k}af(aek)^{m-2-l-k}+\cdots
\end{equation*}
For $(aek)^{l+1}h(aek)^{k}af(aek)^{m-2-l-k}$, we have:
\begin{equation*}
    \begin{split}
        &(aek)^{l+1}h(aek)^{k}af(aek)^{m-2-l-k}\\
        =&(aek)^{l+k+1}afh(aek)^{m-2-l-k}\\
        &-(q-q^{-1})*\sum_{i=1}^{k+1}(aek)^{k+l+1}bfg(aek)^{m-2-l-k}+\cdots\\
    \end{split}
\end{equation*}
Then, using the result of the decomposition of $(aek)^lafh(aek)^{m-1-l}$, we get:
\begin{equation*}
    \begin{split}
        &(aek)^laeh(aek)^kafk(aek)^{m-2-l-k}\\
        =&q^2*(aek)^{l+1}h(aek)^{k}af(aek)^{m-2-l-k}+\cdots\\
        =&q^2*(aek)^{m-1}afh+q^2(q-q^{-1})(m-2-l-k)*(aek)^{m-2}afhbdk\\
        &-q^2(q-q^{-1})(m-1-l)*(aek)^{m-1}bfg\\
        &-q^2(q-q^{-1})(m-2-l-k)*(aek)^{m-1}cdh\\
        &+q^2(q-q^{-1})^2(m-2-l-k)*(aek)^{m-1}ceg+\cdots
    \end{split}
\end{equation*}
\subsubsection{Decomposition of right components}

\hfill

\hfill

\noindent When the left component is $(aek)^lafh(aek)^{m-1-l}$, the corresponding right component is $(ceg)^lchd(ceg)^{m-1-l}$. We have:
\begin{equation*}
    (ceg)^lchd(ceg)^{m-1-l}=cdh(ceg)^{m-1}-(q-q^{-1})*(ceg)^m
\end{equation*}
Recall the decomposition of $(aek)^lafh(aek)^{m-1-l}$ in subsection \ref{2.1.1}. By summing over index $l$, the contribution of $(ceg)^lchd(ceg)^{m-1-l}$ to the linear relation corresponding to comparing basis $\boldsymbol{(aek)^{m-1}afh}$ is:
$$m*cdh(ceg)^{m-1}-m(q-q^{-1})*(ceg)^{m}.$$
The contribution to the linear relation corresponding to comparing basis\\ $\boldsymbol{(aek)^{m-2}afhbdk}$ is:
\begin{equation*}
    \begin{split}
        &(q-q^{-1})\sum_{l=0}^{m-1}(m-1-l)*\left[cdh(ceg)^{m-1}-(q-q^{-1})*(ceg)^m\right]\\
        =&(q-q^{-1})\frac{m(m-1)}{2}*cdh(ceg)^{m-1}-(q-q^{-1})^2\frac{m(m-1)}{2}*(ceg)^m
    \end{split}
\end{equation*}
Similarly, the contribution to the linear relations corresponding to comparing basis $\boldsymbol{(aek)^{m-1}bfg}$ and $\boldsymbol{(aek)^{m-1}cdh}$ are the same:
\begin{equation*}
    \begin{split}
        &-(q-q^{-1})\sum_{l=0}^{m-1}(m-1-l)*\left[cdh(ceg)^{m-1}-(q-q^{-1})*(ceg)^m\right]\\
        =&-(q-q^{-1})\frac{m(m-1)}{2}*cdh(ceg)^{m-1}+(q-q^{-1})^2\frac{m(m-1)}{2}*(ceg)^m
    \end{split}
\end{equation*}
The contribution to the linear relations corresponding to comparing basis\\ $\boldsymbol{(aek)^{m-1}ceg}$ is:
\begin{equation*}
    \begin{split}
        &(q-q^{-1})^2\sum_{l=0}^{m-1}(m-1-l)*\left[cdh(ceg)^{m-1}-(q-q^{-1})*(ceg)^m\right]\\
        =&(q-q^{-1})^2\frac{m(m-1)}{2}*cdh(ceg)^{m-1}-(q-q^{-1})^3\frac{m(m-1)}{2}*(ceg)^m
    \end{split}
\end{equation*}
When the left component is $(aek)^lafk(aek)^kaeh(aek)^{m-2-l-k}$, the corresponding right component is $(ceg)^lchg(ceg)^{k}ced(ceg)^{m-2-l-k}$. We have:
\begin{equation*}
    \begin{split}
        &(ceg)^lchg(ceg)^{k}ced(ceg)^{m-2-l-k}\\
        =&q^{-2k-2}*cdh(ceg)^{m-1}-q^{-2k-2}(q-q^{-1})*(ceg)^m
    \end{split}
\end{equation*}
Recall the decomposition of $(aek)^lafk(aek)^kaeh(aek)^{m-2-l-k}$ in subsection \ref{2.1.1}. We write $\mathcal{F}=cdh(ceg)^{m-1}-(q-q^{-1})*(ceg)^m$. By summing over index $l$ and $k$, the contribution of $(ceg)^lchg(ceg)^{k}ced(ceg)^{m-2-l-k}$ to the linear relation corresponding to comparing basis $\boldsymbol{(aek)^{m-1}afh}$ is:
\begin{equation*}
    \begin{split}
        &\sum_{l=0}^{m-2}\sum_{k=0}^{m-2-l}\left[q^{-2k-2}*cdh(ceg)^{m-1}-q^{-2k-2}(q-q^{-1})*(ceg)^m\right]\\
        =&\left[(m-1)\frac{1}{q^2-1}+\frac{q^{-2m+2}-1}{(1-q^{2})^2}\right]*\mathcal{F}
    \end{split}
\end{equation*}
The contribution to the linear relation corresponding to comparing basis\\ $\boldsymbol{(aek)^{m-2}afhbdk}$ is:
\begin{equation*}
    \begin{split}
        &\sum_{l=0}^{m-2}\sum_{k=0}^{m-2-l}(q-q^{-1})(m-1-l)q^{-2k-2}*\mathcal{F}\\
        =&\left[\frac{m(m-1)}{2q}-\frac{q^{-2}-mq^{-2m}+(m-1)q^{-2m-2}}{q(1-q^{-2})^2}\right]*\mathcal{F}\\
    \end{split}
\end{equation*}
The contribution to the linear relations corresponding to comparing basis $\boldsymbol{(aek)^{m-1}bfg}$ is:
\begin{equation*}
    \begin{split}
        &\sum_{k=0}^{m-2}\sum_{l=0}^{m-2-k}(q-q^{-1})\left[(k+1)q^2-(m-1-l)\right]q^{-2k-2}*\mathcal{F}\\
        =&\frac{(mq^5-2mq^3+mq+q^5+q^3-q)(q^{-2m}-1)}{(1-q^2)^2}*\mathcal{F}\\
        &+\frac{4mq^6-(m^2+5m)q^4+2m^2q^2-m^2+m}{2q(1-q^2)^2}*\mathcal{F}
    \end{split}
\end{equation*}
The contribution to the linear relations corresponding to comparing basis $\boldsymbol{(aek)^{m-1}cdh}$ is:
\begin{equation*}
    \begin{split}
        &-\sum_{l=0}^{m-2}\sum_{k=0}^{m-2-l}(q-q^{-1})(m-1-l)q^{-2k-2}*\mathcal{F}\\
        =&\left[\frac{q^{-2}-mq^{-2m}+(m-1)q^{-2m-2}}{q(1-q^{-2})^2}-\frac{m(m-1)}{2q}\right]*\mathcal{F}\\
    \end{split}
\end{equation*}
The contribution to the linear relations corresponding to comparing basis $\boldsymbol{(aek)^{m-1}ceg}$ is:
\begin{equation*}
    \begin{split}
        &\sum_{k=0}^{m-2}\sum_{l=0}^{m-2-k}(q-q^{-1})^2\left[(m-1-l)-(k+1)q^2\right]q^{-2k-2}*\mathcal{F}\\
        =&\frac{(2mq^6-4mq^4+2mq^2+2q^6+2q^4-2q^2)(q^{-2m}-1)}{2q^2(1-q^2)}*\mathcal{F}\\
        &+\frac{4mq^6-(m^2+5m)q^4+2m^2q^2-m^2+m}{2q^2(1-q^2)}*\mathcal{F}
    \end{split}
\end{equation*}
When the left component is $(aek)^laeh(aek)^kafk(aek)^{m-2-l-k}$, the corresponding right component is $(ceg)^lced(ceg)^{k}chg(ceg)^{m-2-l-k}$. We have:
\begin{equation*}
    \begin{split}
        (ceg)^lced(ceg)^{k}chg(ceg)^{m-2-l-k}=q^{2k}*cdh(ceg)^{m-1}\\
    \end{split}
\end{equation*}
Recall the decomposition of $(aek)^laeh(aek)^kafk(aek)^{m-2-l-k}$ in subsection \ref{2.1.1}. By summing over index $l$ and $k$, the contribution of $(ceg)^lced(ceg)^{k}chg(ceg)^{m-2-l-k}$ to the linear relation corresponding to comparing basis $\boldsymbol{(aek)^{m-1}afh}$ is:
\begin{equation*}
    \begin{split}
        &\sum_{l=0}^{m-2}\sum_{k=0}^{m-2-l}q^{2k+2}*cdh(ceg)^{m-1}\\
        =&\left[(m-1)\frac{1}{q^{-2}-1}+\frac{q^{2m-2}-1}{(1-q^{-2})^2}\right]*cdh(ceg)^{m-1}\\
    \end{split}
\end{equation*}
The contribution to the linear relation corresponding to comparing basis\\ $\boldsymbol{(aek)^{m-2}afhbdk}$ is:
\begin{equation*}
    \begin{split}
        &\sum_{l=0}^{m-2}\sum_{k=0}^{m-2-l}(q-q^{-1})(m-2-l-k)q^{2k+2}*cdh(ceg)^{m-1}\\
        =&\left[\frac{q^{2m-3}-(m-1)q+(m-2)q^{-1}}{(1-q^{-2})^2}-\frac{(m-2)(m-1)q}{2}\right]*cdh(ceg)^{m-1}
    \end{split}
\end{equation*}
The contribution to the linear relations corresponding to comparing basis\\ $\boldsymbol{(aek)^{m-1}bfg}$ is:
\begin{equation*}
    \begin{split}
        &\sum_{l=0}^{m-2}\sum_{k=0}^{m-2-l}-(q-q^{-1})(m-1-l)q^{2k+2}*cdh(ceg)^{m-1}\\
        =&\left[\frac{m(m-1)q}{2}-\frac{q^3-mq^{2m+1}+(m-1)q^{2m+3}}{(1-q^2)^2}\right]*cdh(ceg)^{m-1}
    \end{split}
\end{equation*}
The contribution to the linear relations corresponding to comparing basis\\ $\boldsymbol{(aek)^{m-1}cdh}$ is:
\begin{equation*}
    \begin{split}
        &\sum_{l=0}^{m-2}\sum_{k=0}^{m-2-l}-(q-q^{-1})(m-2-l-k)q^{2k+2}*cdh(ceg)^{m-1}\\
        =&\left[\frac{(m-2)(m-1)q}{2}-\frac{q^{2m-3}-(m-1)q+(m-2)q^{-1}}{(1-q^{-2})^2}\right]*cdh(ceg)^{m-1}
    \end{split}
\end{equation*}
The contribution to the linear relations corresponding to comparing basis\\ $\boldsymbol{(aek)^{m-1}ceg}$ is:
\begin{equation*}
    \begin{split}
        &\sum_{l=0}^{m-2}\sum_{k=0}^{m-2-l}(q-q^{-1})^2(m-2-l-k)q^{2k+2}*cdh(ceg)^{m-1}\\
        =&\left[\frac{q^{2m-2}-(m-1)q^2+(m-2)}{(1-q^{-2})}-\frac{(m-2)(m-1)(q^2-1)}{2}\right]*cdh(ceg)^{m-1}
    \end{split}
\end{equation*}
\subsubsection{Contribution to linear relations corresponding to different comparing basis}

\hfill

\hfill

\noindent The contribution of the 3 types of tensor products determines the linear relation corresponding to comparing basis $(aek)^{m-1}afh$. The coefficient of $cdh(ceg)^{m-1}$ is:
\begin{equation*}
    \begin{split}
        &m+\frac{m-1}{q^2-1}+\frac{q^{-2m+2}-1}{(1-q^2)^2}+\frac{m-1}{q^{-2}-1}+\frac{q^{2m-2}-1}{(1-q^{-2})^2}\\
        =&\frac{q^2(q^m-q^{-m})^2}{(1-q^2)^2}
    \end{split}
\end{equation*}
When computing the coefficient of $(ceg)^m$, we have to consider the coefficient of $(aek)^{m-1}afh$ in $D_q^m$. By Theorem 1 e), terms in $D_q^m$ whose decomposition contains  $(aek)^{m-1}afh$ has to be in the form of $(aek)^lafh(aek)^{m-l-1}$. Thus, the coefficient of $(aek)^{m-1}afh$ in $D_q^m$ is $-mq$. Then, the coefficient of $(ceg)^{m}$ is:
\begin{equation*}
    \begin{split}
        &-m(q-q^{-1})+mq-(q-q^{-1})*\left(\frac{m-1}{q^2-1}+\frac{q^{-2m+2}-1}{(1-q^2)^2}\right)\\
        =&\frac{1}{q}-\frac{q^{-2m+2}-1}{q(q^2-1)}=\frac{q(1-q^{-2m})}{(q^2-1)}
    \end{split}
\end{equation*}
The contribution of the 3 types of tensor products to the coefficient of $cdh(ceg)^{m-1}$ in the linear relation corresponding to comparing basis $(aek)^{m-2}afhbdk$ is:
\begin{equation*}
    \begin{split}
        &(q-q^{-1})\frac{m(m-1)}{2}+\frac{m(m-1)}{2q}-\frac{q^{-2}-mq^{-2m}+(m-1)q^{-2m-2}}{q(1-q^{-2})^2}\\
        &+\frac{q^{2m-3}-(m-1)q+(m-2)q^{-1}}{(1-q^{-2})^2}-\frac{(m-2)(m-1)q}{2}\\
        =&\frac{(1-q^{2m})(mq^2-q^{2m}-m+1)}{q^{2m-1}(1-q^2)^2}
    \end{split}
\end{equation*}
The contribution of the 3 types of tensor products to the coefficient of $(ceg)^{m}$ in the linear relation corresponding to comparing basis $(aek)^{m-2}afhbdk$ is:
\begin{equation*}
    \begin{split}
        &-(q-q^{-1})^2\frac{m(m-1)}{2}-(q-q^{-1})\frac{m(m-1)}{2q}\\
        &+(q-q^{-1})\frac{q^{-2}-mq^{-2m}+(m-1)q^{-2m-2}}{q(1-q^{-2})^2}\\
        =&-(q^2-1)\frac{m(m-1)}{2}+\frac{q^{-2}-mq^{-2m}+(m-1)q^{-2m-2}}{1-q^{-2}}\\
    \end{split}
\end{equation*}
The contribution of the 3 types of tensor products to the coefficient of $cdh(ceg)^{m-1}$ in the linear relation corresponding to comparing basis $(aek)^{m-2}bfg$ is:
\begin{equation*}
    \begin{split}
        &-(q-q^{-1})\frac{m(m-1)}{2}+\frac{m(m-1)q}{2}-\frac{q^3-mq^{2m+1}+(m-1)q^{2m+3}}{(1-q^2)^2}\\
        &+\frac{(2mq^6-4mq^4+2mq^2+2q^6+2q^4-2q^2)(q^{-2m}-1)}{2q(1-q^2)^2}\\
        &+\frac{4mq^6-(m^2+5m)q^4+2m^2q^2-m^2+m}{2q(1-q^2)^2}\\
        =&\frac{2mq^5-(3m+1)q^3+mq}{(1-q^2)^2}+\frac{mq^{2m+1}-(m-1)q^{2m+3}}{(1-q^2)^2}\\
        &+\frac{(mq^5-2mq^3+mq+q^5+q^3-q)(q^{-2m}-1)}{(1-q^2)^2}\\
    \end{split}
\end{equation*}
The contribution of the 3 types of tensor products to the coefficient of $(ceg)^{m}$ in the linear relation corresponding to comparing basis $(aek)^{m-1}bfg$ is:
\begin{equation*}
    \begin{split}
        &(q-q^{-1})^2\frac{m(m-1)}{2}-(q-q^{-1})\frac{4mq^6-(m^2+5m)q^4+2m^2q^2-m^2+m}{2q(1-q^2)^2}\\
        &-(q-q^{-1})\frac{(mq^5-2mq^3+mq+q^5+q^3-q)(q^{-2m}-1)}{(1-q^2)^2}\\
        =&\frac{(5m-m^2)q^4+(2m^2-8m)q^2+(3m-m^2)}{2(1-q^2)}\\
        &+\frac{(mq^4-2mq^2+m+q^4+q^2-1)(q^{-2m}-1)}{1-q^2}\\
    \end{split}
\end{equation*}
Notice that the sign of the coefficient of $(aek)^{m-1}cdh$ is always the opposite as that of $(aek)^{m-2}afhbdk$ in the decomposition of all 3 possible forms of right components. The contribution of the 3 types of tensor products to the coefficient of $cdh(ceg)^{m-1}$ in the linear relation corresponding to comparing basis $(aek)^{m-1}cdh$ is:
\begin{equation*}
    \begin{split}
        &-(q-q^{-1})\frac{m(m-1)}{2}+\frac{q^{-2}-mq^{-2m}+(m-1)q^{-2m-2}}{q(1-q^{-2})^2}-\frac{m(m-1)}{2q}\\
        &+\frac{(m-2)(m-1)q}{2}-\frac{q^{2m-3}-(m-1)q+(m-2)q^{-1}}{(1-q^{-2})^2}\\
        =&-\frac{q^{2m-3}-mq^{-1}+(m-2)q^{-3}+mq^{-2m-1}-(m-1)q^{-2m-3}}{(1-q^{-2})^2}
    \end{split}
\end{equation*}
The contribution of the 3 types of tensor products to the coefficient of $(ceg)^{m}$ in the linear relation corresponding to comparing basis $(aek)^{m-1}cdh$ is:
\begin{equation*}
    \begin{split}
        &(q-q^{-1})^2\frac{m(m-1)}{2}+(q-q^{-1})\frac{m(m-1)}{2q}\\
        &-(q-q^{-1})\frac{q^{-2}-mq^{-2m}+(m-1)q^{-2m-2}}{q(1-q^{-2})^2}\\
        &=(q^2-1)\frac{m(m-1)}{2}-\frac{q^{-2}-mq^{-2m}+(m-1)q^{-2m-2}}{1-q^{-2}}
    \end{split}
\end{equation*}
The contribution of the 3 types of tensor products to the coefficient of $cdh(ceg)^{m-1}$ in the linear relation corresponding to comparing basis $(aek)^{m-1}ceg$ is:
\begin{equation*}
    \begin{split}
        &(q-q^{-1})^2\frac{m(m-1)}{2}+\frac{(2mq^6-4mq^4+2mq^2+2q^6+2q^4-2q^2)(q^{-2m}-1)}{2q^2(1-q^2)}\\
        &+\frac{4mq^6-(m^2+5m)q^4+2m^2q^2-m^2+m}{2q^2(1-q^2)}\\
        &+\frac{q^{2m-2}-(m-1)q^2+(m-2)}{(1-q^{-2})}-\frac{(m-2)(m-1)(q^2-1)}{2}\\
        =&\frac{(mq^4-2mq^2+m+q^4+q^2-1)(q^{-2m}-1)}{(1-q^2)}+\frac{-2mq^4+2mq^2+q^{2m}-1}{q^2-1}
    \end{split}
\end{equation*}
The contribution of the 3 types of tensor products to the coefficient of $(ceg)^{m}$ in the linear relation corresponding to comparing basis $(aek)^{m-1}ceg$ is:
\begin{equation*}
    \begin{split}
        &-(q-q^{-1})^3\frac{m(m-1)}{2}-(q-q^{-1})\frac{4mq^6-(m^2+5m)q^4+2m^2q^2-m^2+m}{2q^2(1-q^2)}\\
        &-(q-q^{-1})\frac{(mq^4-2mq^2+m+q^4+q^2-1)(q^{-2m}-1)}{(1-q^2)}\\
        =&\frac{(3m-m^2-2)q^3}{2}+(m^2-2m-1)q+\frac{m+2-m^2}{2q}+(m+1)q^{-2m+3}\\
        &+(1-2m)q^{-2m+1}+(m-1)q^{-2m-1}
    \end{split}
\end{equation*}

\subsection{The recursive relation for the Haar state of standard monomials in the form of $(cdh)^i(ceg)^{m-i}$}

\hfill

\hfill

\noindent In the following, we will derive the recursive relation of $h\left((cdh)^i(ceg)^{m-i}\right)$. The case $i=1$ is solved in the Source Matrix. We will start with the general case $3\le i\le m-1$ and then the case $i=2, m$.  

\subsubsection{Recursive relation of $3\le i\le m-1$}

\hfill

\hfill

\noindent We will use equation basis $(cdh)^{i-1}(ceg)^{m-i+1}$ and comparing basis $(aek)^{m-1}afh$ to derive the recursive relation of the Haar state of $(cdh)^{i}(ceg)^{m-i}$. In the comultiplication of $(cdh)^{i-1}(ceg)^{m-i+1}$, the left components containing $(aek)^{m-1}afh$ are
\begin{enumerate}
    \item[1)] $(aek)^lafh(aek)^{m-1-l}$
    \item[2)] $(aek)^lafk(aek)^kaeh(aek)^{m-2-l-k}$
    \item[3)] $(aek)^laeh(aek)^kafk(aek)^{m-2-l-k}$ 
\end{enumerate}
\noindent When the left component is in the form $(aek)^lafh(aek)^{m-1-l}$, the coefficient of $(aek)^{m-1}afh$ in the decomposition of $(aek)^lafh(aek)^{m-1-l}$ is $1$ and the corresponding relation components are:
\begin{itemize}
    \item[1)] $(cdh)^lcge(cdh)^{i-2-l}(ceg)^{m-i+1}$ 
    \item[2)] $(cdh)^{i-1}(ceg)^lchd(ceg)^{m-i-l}$
\end{itemize}
When the left component is in the form $(aek)^lafk(aek)^kaeh(aek)^{m-2-l-k}$, the coefficient of $(aek)^{m-1}afh$ in the decomposition of $(aek)^lafk(aek)^kaeh(aek)^{m-2-l-k}$ is $1$ and the corresponding relation components are:
\begin{itemize}
    \item[3)] $(cdh)^lcgh(cdh)^kcde(cdh)^{i-3-l-k}(ceg)^{m-i+1}$
    \item[4)] $(cdh)^lcgh(cdh)^{i-2-l}(ceg)^kced(ceg)^{m-i-k}$
    \item[5)] $(cdh)^{i-1}(ceg)^lchg(ceg)^kced(ceg)^{m-i-1-l-k}$
\end{itemize}
When the left component is in the form $(aek)^laeh(aek)^kafk(aek)^{m-2-l-k}$, the coefficient of $(aek)^{m-1}afh$ in the decomposition of $(aek)^laeh(aek)^kafk(aek)^{m-2-l-k}$ is $q^2$ and the corresponding relation components are:
\begin{itemize}
    \item[6)] $(cdh)^lcde(cdh)^kcgh(cdh)^{i-3-l-k}(ceg)^{m-i+1}$
    \item[7)] $(cdh)^lcde(cdh)^{i-2-l}(ceg)^kchg(ceg)^{m-i-k}$
    \item[8)] $(cdh)^{i-1}(ceg)^lced(ceg)^kchg(ceg)^{m-i-1-l-k}$
\end{itemize}
For case 1), we have:
\begin{equation*}
    (cdh)^lcge(cdh)^{i-2-l}(ceg)^{m-i+1}=(cdh)^{i-2}(ceg)^{m-i+2}
\end{equation*}
The contribution of case 1) to the linear relation is:
\begin{equation*}
    (i-1)*(cdh)^{i-2}(ceg)^{m-i+2}
\end{equation*}
For case 2), we have:
\begin{equation*}
    (cdh)^{i-1}(ceg)^lchd(ceg)^{m-i-l}=(cdh)^{i}(ceg)^{m-i}-(q-1/q)*(cdh)^{i-1}(ceg)^{m-i+1}
\end{equation*}
The contribution of case 2) to the linear relation is:
\begin{equation*}
    (m-i+1)*(cdh)^{i}(ceg)^{m-i}-(m-i+1)(q-1/q)*(cdh)^{i-1}(ceg)^{m-i+1}
\end{equation*}
For case 3), we have:
\begin{equation*}
    \begin{split}
        &(cdh)^lcgh(cdh)^kcde(cdh)^{i-3-l-k}(ceg)^{m-i+1}\\
        &=(cdh)^l(chd)^{k+1}(cdh)^{i-3-l-k}(ceg)^{m-i+2}\\
        &=(cdh)^l(cdh-(q-1/q)*ceg)^{k+1}(cdh)^{i-3-l-k}(ceg)^{m-i+2}\\
        &=\sum_{j=0}^{k+1}(-(q-1/q))^j{k+1\choose j}*(cdh)^{i-2-j}(ceg)^{m-i+2+j}
    \end{split}
\end{equation*}
The contribution of case 3) to the linear relation is:
\begin{equation*}
\begin{split}
    &\sum_{l=0}^{i-3}\sum_{k=0}^{i-3-l}\sum_{j=0}^{k+1}(-(q-1/q))^j{k+1\choose j}*(cdh)^{i-2-j}(ceg)^{m-i+2+j}\\
    &+\sum_{l=0}^{i-3}\sum_{j=1}^{i-2-l}(-(q-1/q))^j{i-1-l\choose j+1}*(cdh)^{i-2-j}(ceg)^{m-i+2+j}\\
    =&\frac{(i-1)(i-2)}{2}*(cdh)^{i-2}(ceg)^{m-i+2}\\
    &+\sum_{k=3}^{i}(-(q-1/q))^{k-2}{i\choose k}*(cdh)^{i-k}(ceg)^{m-i+k}\\
\end{split}
\end{equation*}
For case 4), we have:
\begin{equation*}
    \begin{split}
        &(cdh)^lcgh(cdh)^{i-2-l}(ceg)^kced(ceg)^{m-i-k}\\
        &=q^{-2k-1}*(cdh)^l(chd)^{i-1-l}(ceg)^{m-i+1}\\
        &=q^{-2k-1}\sum_{j=0}^{i-1-l}(-(q-1/q))^j{i-1-l\choose j}*(cdh)^{i-1-j}(ceg)^{m-i+1+j}\\
    \end{split}
\end{equation*}
The contribution of case 4) to the linear relation is:
\begin{equation*}
\begin{split}
    &\sum_{k=0}^{m-i}q^{-2k-1}\sum_{l=0}^{i-2}\sum_{j=0}^{i-1-l}(-(q-1/q))^j{i-1-l\choose j}*(cdh)^{i-1-j}(ceg)^{m-i+1+j}\\
    &\sum_{k=0}^{m-i}q^{-2k-1}\sum_{l=1}^{i-1}\sum_{j=0}^{l}(-(q-1/q))^j{l\choose j}*(cdh)^{i-1-j}(ceg)^{m-i+1+j}\\
    =&q^{-1}\frac{1-q^{-2(m-i+1)}}{1-q^{-2}}((i-1)*(cdh)^{i-1}(ceg)^{m-i+1}\\
    &+\sum_{k=2}^{i}(-(q-1/q))^{k-1}{i \choose k}*(cdh)^{i-k}(ceg)^{m-i+k})
\end{split}
\end{equation*}
For case 5), we have:
\begin{equation*}
    \begin{split}
        &(cdh)^{i-1}(ceg)^lchg(ceg)^kced(ceg)^{m-i-1-l-k}\\
        =&(cdh)^{i-1}(ceg)^lch(ceg)^{k+1}d(ceg)^{m-i-1-l-k}\\
        =&q^{-2k-2}*(cdh)^{i-1}chd(ceg)^{m-i}\\
        =&q^{-2k-2}*((cdh)^{i}(ceg)^{m-i}-(q-1/q)*(cdh)^{i-1}(ceg)^{m-i+1})
    \end{split}
\end{equation*}
The contribution of case 5) to the linear relation is:
\begin{equation*}
\begin{split}
    &\sum_{l=0}^{m-i-1}\sum_{k=0}^{m-i-1-l}q^{-2k-2}*((cdh)^{i}(ceg)^{m-i}-(q-1/q)*(cdh)^{i-1}(ceg)^{m-i+1})\\
    =&(\frac{m-i}{q^2-1}+\frac{q^{-2(m-i)}-1}{(1-q^2)^2})*((cdh)^{i}(ceg)^{m-i}-(q-1/q)*(cdh)^{i-1}(ceg)^{m-i+1})
\end{split}
\end{equation*}
For case 6), we have:
\begin{equation*}
    \begin{split}
        &(cdh)^lcde(cdh)^kcgh(cdh)^{i-3-l-k}(ceg)^{m-i+1}\\
        =&q^2*(cdh)^lcd(cdh)^kh(cdh)^{i-3-l-k}(ceg)^{m-i+2}\\
        =&q^2*(cdh)^lcd(hcd+(q+1/q)*ceg)^kh(cdh)^{i-3-l-k}(ceg)^{m-i+2}\\
        =&q^2\sum_{j=0}^k(q-1/q)^j{k \choose j}*(cdh)^{l+k-j}cd(ceg)^jh(cdh)^{i-3-l-k}(ceg)^{m-i+2}\\
        =&q^2\sum_{j=0}^k(q-1/q)^jq^{2j}{k \choose j}*(cdh)^{i-2-j}(ceg)^{m-i+2+j}\\
    \end{split}
\end{equation*}
The contribution of case 6) to the linear relation is:
\begin{equation*}
\begin{split}
    &q^4\sum_{l=0}^{i-3}\sum_{k=0}^{i-3-l}\sum_{j=0}^k(q-1/q)^jq^{2j}{k \choose j}*(cdh)^{i-2-j}(ceg)^{m-i+2+j}\\
    &=\sum_{k=2}^{i-1}(q-1/q)^{k-2}q^{2k}{i-1\choose k}*(cdh)^{i-k}(ceg)^{m-i+k}
\end{split}
\end{equation*}
For case 7), we have:
\begin{equation*}
    \begin{split}
        &(cdh)^lcde(cdh)^{i-2-l}(ceg)^kchg(ceg)^{m-i-k}\\
        =&q^{2k+1}\sum_{j=0}^{i-2-l}(q-1/q)^jq^{2j}{i-2-l\choose j}*(cdh)^{i-1-j}(ceg)^{m-i+1+j}\\
    \end{split}
\end{equation*}
The contribution of case 7) to the linear relation is:
\begin{equation*}
\begin{split}
    &q^2\sum_{k=0}^{m-i}q^{2k+1}\sum_{l=0}^{i-2}\sum_{j=0}^{i-2-l}(q-1/q)^jq^{2j}{i-2-l\choose j}*(cdh)^{i-1-j}(ceg)^{m-i+1+j}\\
     =&q\frac{1-q^{2(m-i+1)}}{1-q^2}\sum_{k=1}^{i-1}(q-1/q)^{k-1}q^{2k}{i-1\choose k}*(cdh)^{i-k}(ceg)^{m-i+k}
\end{split}
\end{equation*}
For case 8), we have:
\begin{equation*}
    \begin{split}
        &(cdh)^{i-1}(ceg)^lced(ceg)^kchg(ceg)^{m-i-1-l-k}\\
        =&q^{2k}*(cdh)^{i}(ceg)^{m-i}
    \end{split}
\end{equation*}
The contribution of case 8) to the linear relation is:
\begin{equation*}
\begin{split}
    &q^2\sum_{l=0}^{m-i-1}\sum_{k=0}^{m-i-1-l}q^{2k}*(cdh)^{i}(ceg)^{m-i}\\
    =&\left(\frac{(m-i)q^2}{1-q^2}+\frac{q^4(q^{2(m-i)}-1)}{(1-q^{2})^2}\right)*(cdh)^{i}(ceg)^{m-i}
\end{split}
\end{equation*}
The term $(cdh)^{i}(ceg)^{m-i}$ appears in case 2), 5), and 8). Summing the contributions from these cases, the coefficient of $(cdh)^{i}(ceg)^{m-i}$ is:
\begin{equation*}
    \begin{split}
        &\frac{q^2(q^{m-i+1}-q^{-(m-i+1)})^2}{(1-q^2)^2}
    \end{split}
\end{equation*}
The term $(cdh)^{i-1}(ceg)^{m-i+1}$ appears in case 2), 4), 5), 7), and the right-hand side of Equation (\ref{eq:3}). Summing the contributions from these cases, the coefficient of $(cdh)^{i-1}(ceg)^{m-i+1}$ is:
\begin{equation*}
    \begin{split}
        &\frac{-(i-1)q^{2(m-i)+5}+iq^{-2(m-i)-1}-q}{(1-q^2)}
    \end{split}
\end{equation*}
The term $(cdh)^{i-2}(ceg)^{m-i+2}$ appears in case 1), 3), 4), 6), and 7). Notice that if we combine the contribution of case 1) and case 3), we get:
\begin{equation*}
    (i-1)+\frac{(i-1)(i-2)}{2}=\frac{i(i-1)}{2}={i\choose 2}
\end{equation*}
which corresponds to $k=2$ in the summation of case 3). Thus, we can treat $(cdh)^{i-2}(ceg)^{m-i+2}$ in the same way as the general case $(cdh)^{i-k}(ceg)^{m-i+k}$, $3\le k\le i-1$ which appears in case 3), 4), 6), and 7). Summing the contributions from these cases, the coefficient of $(cdh)^{i-k}(ceg)^{m-i+k}$ for $2\le k\le i-1$ is:
\begin{equation*}
    \begin{split}
        &(q^{-1}-q)^{k-2}{i\choose k}q^{-2(m-i+1)}+(q-q^{-1})^{k-2}q^{2k}{i-1\choose k}q^{2(m-i+1)}
    \end{split}
\end{equation*}
Notice that if we put $k=1$ in the above coefficient, we get:
\begin{equation*}
    \begin{split}
        &(q^{-1}-q)^{-1}iq^{-2(m-i+1)}+(q-q^{-1})^{-1}q^2(i-1)q^{2(m-i+1)}\\
        =&\frac{iq^{-2m+2i-1}-(i-1)q^{2m-2i+5}}{1-q^2}
    \end{split}
\end{equation*}
The term $(ceg)^m$ appears in case 3) and 4). Summing the contributions from these cases, the coefficient of $(ceg)^{m}$ is:
\begin{equation*}
\begin{split}
    &(q^{-1}-q)^{i-2}q^{-2(m-i+1)}
\end{split}
\end{equation*}
The expressions of the coefficients of $(cdh)^{i-2}(ceg)^{m-i+2}$ and $(cdh)^{i-k}(ceg)^{m-i+k}$ for $3\le k\le i-1$ are consistent. Thus, the expression of $(cdh)^{i}(ceg)^{m-i}$ is:
\begin{equation}
    \begin{split}
         &\frac{q^2(q^{m-i+1}-q^{-(m-i+1)})^2}{(1-q^2)^2}*h((cdh)^i(ceg)^{m-i})\\
         =&-\frac{q}{q^2-1}*h((cdh)^{i-1}(ceg)^{m-i+1})\\     
         &-\sum_{k=1}^{i-1}c_k*h((cdh)^{i-k}(ceg)^{m-i+k})-(q^{-1}-q)^{i-2}q^{-2(m-i+1)}h((ceg)^{m})\\
    \end{split}\label{recursive:1}
\end{equation}
where
\begin{equation*}
\begin{split}
    c_k=(q^{-1}-q)^{k-2}{i\choose k}q^{-2(m-i+1)}+(q-1/q)^{k-2}q^{2k}{i-1\choose k}q^{2(m-i+1)}.
\end{split}
\end{equation*}
\subsubsection{Special case $i=2$}\label{sp_cs_i=2}

\hfill

\hfill

\noindent We use the linear relation derived from $cdh(ceg)^{m-1}$ and equation basis $(aek)^{m-1}afh$. The situation in this subsection is similar to the previous subsection but we do not have case 3) and 6) anymore. The linear relation in this case only involves $(cdh)^2(ceg)^{m-2}$, $cdh(ceg)^{m-1}$, and $(ceg)^m$. Here, $(ceg)^m$ is the same as $(cdh)^{i-2}(ceg)^{m-i+2}$ in the previous subsection. If we substitute $i=2$ into the contribution of case 3) and 6), we find that the contributions corresponding the two cases are automatically zero. Thus, recursive relation Equation (\ref{recursive:1}) is consistent with the case in subsection \ref{sp_cs_i=2}.

\subsubsection{Special case $i=m$}\label{sp_cs_i=m}

\hfill

\hfill

\noindent We use the linear relation derived from $(cdh)^{m-1}ceg$ and equation basis $(aek)^{m-1}afh$. The situation in this subsection is similar to the previous subsection but we do not have case 5) and 8) anymore. Thus, the coefficient of $(cdh)^m$ is $1$ and the coefficient of $(cdh)^{m-1}ceg$ is:
\begin{equation*}
    \begin{split}
        &-(q-1/q)+q^{-1}(m-1)+q^3(m-1)+mq\\
        &=mq^{-1}+(m-1)q+(m-1)q^3
    \end{split}
\end{equation*}
On the other hand, if we substitute $i=m$ in to the coefficient of $(cdh)^{i}(ceg)^{m-i}=(cdh)^{m}$, we get $1$ and coefficient of $(cdh)^{i-1}(ceg)^{m-i+1}=(cdh)^{m-1}ceg$, we get:
\begin{equation*}
    \begin{split}
        &\frac{mq^{-1}-(m-1)q^{5}-q}{1-q^2}\\
        &=(m-1)q^3+(m-1)q+mq^{-1}
    \end{split}
\end{equation*}
Thus, recursive relation Equation (\ref{recursive:1}) is consistent with the case in subsection \ref{sp_cs_i=m}.

\subsection{The recursive relation for the Haar state of standard monomials in the form of $(bfg)^i(ceg)^{m-i}$}
\hfill

\hfill

\noindent  We will use equation basis $(bfg)^{i-1}(ceg)^{m-i+1}$ and comparing basis $(aek)^{m-1}bdk$ to derive the recursive relation for $h\left((bfg)^i(ceg)^{m-i}\right)$. The computation involved in finding this recursive relation is similar to that of finding the recursive relation of $h\left((cdh)^i(ceg)^{m-i}\right)$ and the derived recursive relation of $h\left((bfg)^i(ceg)^{m-i}\right)$ is the same as Equation (\ref{recursive:1}).

\section[General algorithm to compute the Haar states of standard monomials on $\mathcal{O}(SL_q(3))$]{General algorithm to compute the Haar states of standard monomials on $\mathcal{O}(SL_q(3))$}
\noindent In this section, we assume that the Haar states of all standard monomials of order $m-1$ are known and we want to compute the Haar states of all standard monomials of order $m$. For the simplicity of our argument, we will show that our proposed algorithm is able to compute the Haar state of all standard monomials, not just the monomials basis we picked.

\subsection{Zigzag recursive pattern for standard monomials in the form\\ $(bdk)^r(bfg)^{s}(cdh)^j(ceg)^{m-i-j}$ and $(afh)^r(cdh)^{s}(bfg)^j(ceg)^{m-i-j}$ with $r+s=1$}

\hfill

\hfill

We start with the Haar state of $(bdk)^r(bfg)^{s}(cdh)^j(ceg)^{m-i-j}$, $r+s=1$. In this case, we compute the Haar state of monomials in form $bdk(cdh)^{j-1}(ceg)^{m-j}$ and $bfg(cdh)^{j-1}(ceg)^{m-j}$. We use an induction on the value $j$. We know the Haar state for case $j=1$ from the solution of the source matrix of order $m$. For $j=2$, the Haar state of $bfgcdh(ceg)^{m-2}$ is know as well. To compute the Haar state of $bdkcdh(ceg)^{m-2}$, we use the linear relation derived from equation basis $bdk(ceg)^{m-1}$ and comparing basis $(aek)^{m-1}afh$. Assume we have solved all the Haar state for $j\le t-1$. Then, we can compute the Haar state of $bfg(cdh)^t(ceg)^{m-t-1}$ by the linear relation derived from equation basis $bfg(cdh)^{t-1}(ceg)^{m-t}$ and comparing basis $(aek)^{m-1}afh$. Next, we compute the Haar state of of $bdk(cdh)^t(ceg)^{m-t-1}$ by the linear relation derived from equation basis $bdk(cdh)^{t-1}(ceg)^{m-t}$ and comparing basis $(aek)^{m-1}afh$. During the process, the only monomial with unknown Haar state appearing in the linear relation is the monomial which we are pursuing. The order that we used to compute these Haar states are depicted in Appendix \ref{apd:d}. Since we solve the Haar state of ${bfg(cdh)^j(ceg)^{m-1-j}}$ and ${bdk(cdh)^j(ceg)^{m-1-j}}$ in a ``zigzag'' pattern in the figure, we call this recursive relation the {\bf Zigzag recursive relation}.

\hfill

\noindent We can compute the Haar states of monomials in form $afh(bfg)^{j}(ceg)^{m-j-1}$ and $cdh(bfg)^{j}(ceg)^{m-j-1}$ in the same order. When we derive linear relations, we use equation basis $afh(bfg)^{j-1}(ceg)^{m-j}$ to substitute $bdk(cdh)^{j-1}(ceg)^{m-j}$ and $cdh(bfg)^{j-1}(ceg)^{m-j}$ to substitute $bfg(cdh)^{j-1}(ceg)^{m-j}$ and use comparing basis $(aek)^{m-1}bdk$ to substitute $(aek)^{m-1}afh$.

\subsection{Standard monomials ending with $\boldsymbol{(ceg)^{m-2}}$ and standard monomials ending with $\boldsymbol{bfgbfg(ceg)^{m-3}}$, $\boldsymbol{cdhcdh(ceg)^{m-3}}$, and $\boldsymbol{bfgcdh(ceg)^{m-3}}$.}

\subsubsection{Standard monomials ending with $cdh(ceg)^{m-2}$ or $bfg(ceg)^{m-2}$}

\hfill \\

First, notice that if we choose $(bfg)^2(ceg)^{m-2}$ as the equation basis and use $(aek)^{m-1}afh$ as the comparing basis, the derived linear relation only includes the Haar state of $bdkbfg(ceg)^{m-2}$, $cdh(bfg)^2(ceg)^{m-3}$, $(bfg)^2(ceg)^{m-2}$, and $bfg(ceg)^{m-1}$. Combining the results from previous subsections, we can find the Haar state of ${bdkbfg(ceg)^{m-2}}$.

\hfill

\noindent Similarly, if we choose $(cdh)^2(ceg)^{m-2}$ as the equation basis and use $(aek)^{m-1}bdk$ as the comparing basis, the derived linear relation only includes the Haar state of $afhcdh(ceg)^{m-2}$, $bfg(cdh)^2(ceg)^{m-3}$, $(cdh)^2(ceg)^{m-2}$, and $cdh(ceg)^{m-1}$. Combining the results from previous subsections, we can find the Haar state of ${afhcdh(ceg)^{m-2}}$.

\hfill

\noindent Next, we consider equation bases $aek(ceg)^{m-1}$ with comparing basis $(aek)^{m-1}afh$ and $(aek)^{m-1}bdk$. The linear relation derived by comparing basis $(aek)^{m-1}afh$ includes $aekcdh(ceg)^{m-2}$, $afhcdh(ceg)^{m-2}$, $aek(ceg)^{m-1}$, and $afh(ceg)^{m-1}$. Thus, we can solve the Haar state of ${aekcdh(ceg)^{m-2}}$ from this linear relation. Similarly, The linear relation derived by comparing basis $(aek)^{m-1}bdk$ includes $aekbfg(ceg)^{m-2}$, $afhbfg(ceg)^{m-2}$, $aek(ceg)^{m-1}$, and $afh(ceg)^{m-1}$. Thus, we can solve the Haar state of ${aekbfg(ceg)^{m-2}}$ from this linear relation.

\subsubsection{Standard monomials ending with $cdhcdh(ceg)^{m-3}$ or $bfgbfg(ceg)^{m-3}$}

\hfill

\hfill

First, we consider standard monomials ending with $cdhcdh(ceg)^{m-3}$. They are $ceg(cdh)^2(ceg)^{m-3}$ $(cdh)^3(ceg)^{m-3}$, $bfg(cdh)^2(ceg)^{m-3}$, $bdk(cdh)^2(ceg)^{m-3}$,\\ $afh(cdh)^2(ceg)^{m-3}$, and $aek(cdh)^2(ceg)^{m-3}$. Here, the Haar states of $afh(cdh)^2(ceg)^{m-3}$ and $aek(cdh)^2(ceg)^{m-3}$ are still unknown. To compute the Haar states of the two monomial, we construct a $2\times 2$ linear system consisting of:
\begin{itemize}
    \item[1)] the quantum determinant condition: $D_q*(cdh)^2(ceg)^{m-3}=(cdh)^2(ceg)^{m-3}$;
    \item[2)] the linear relation derived from equation basis $afhcdh(ceg)^{m-2}$ and comparing basis $(aek)^mafh$.
\end{itemize}
Similarly, we need to compute the Haar state of monomials $bdk(bfg)^2(ceg)^{m-3}$ and $aek(bfg)^2(ceg)^{m-3}$. We construct a $2\times 2$ linear system consisting of:
\begin{itemize}
    \item[1)] the quantum determinant condition: $D_q*(bfg)^2(ceg)^{m-3}=(bfg)^2(ceg)^{m-3}$;
    \item[2)] the linear relation derived from equation basis $bdkbfg(ceg)^{m-2}$ and comparing basis $(aek)^mbdk$.
\end{itemize}

\subsubsection{Matrix of $aekbfgcdh(ceg)^{m-3}$, $afhbfgcdh(ceg)^{m-3}$, and $bdkbfgcdh(ceg)^{m-3}$}

\hfill \\

To compute the Haar state of these three monomials, we construct a $3\times 3$ linear system consisting of:
\begin{itemize}
    \item[1)] the quantum determinant condition: $D_q*bfgcdh(ceg)^{m-3}=bfgcdh(ceg)^{m-3}$;
    \item[2)] the linear relation derived from equation basis $afhbfg(ceg)^{m-2}$ and comparing basis $(aek)^{m-1}afh$;
    \item[3)] the linear relation derived from equation basis $bdkcdh(ceg)^{m-2}$ and comparing basis $(aek)^{m-1}bdk$.
\end{itemize}
Entries of the system matrix are listed below:

\begin{table}[htb!]
    \centering
    \caption{Entries in the system matrix}
    \resizebox{1\textwidth}{!}{
    \begin{tabular}{|c|c|c|c|}
    \hline
    \diagbox{Relation}{Haar State} &$aekbfgcdh(ceg)^{m-3}$ &$afhbfgcdh(ceg)^{m-3}$ & $bdkbfgcdh(ceg)^{m-3}$ \\ \hline
    Quantum Determinant &1 &$-q$ &$-q$  \\ \hline
     $afhbfg(ceg)^{m-2}$&$(q+1/q)\frac{q^2-q^{-2(m-2)}}{q^2-1}$ &$\frac{(q^{-2(m-1)}-1)(q^6-q^2-q^{2m}+1)}{(q^2-1)^2}$ & 0 \\ \hline
     $bdkcdh(ceg)^{m-2}$&$q+1/q+q(q^2+1)\frac{1-q^{2(m-2)}}{1-q^{2}}$ &$(q^2+1)(q^{-2}-q^{2(m-2)})$ &$\frac{q^{-2(m-2)}+q^{2(m-1)}-q^2-1}{(q^2-1)^2}$  \\ \hline
    \end{tabular}
    }
    \label{tab:my_label}
\end{table}
\noindent Using Gauss elimination, we have:
\begin{equation*}
    \begin{bmatrix}
    1&-q&-q\\
    0&\frac{q^{-2(m-1)}+q^{2m}-q^2-1}{(q^2-1)^2}&\frac{(q^2+1)(q^{-2(m-2)}-q^2)}{1-q^2}\\
    0&\frac{(q^2+1)(q^{-2}-q^{2(m-2)})}{1-q^2}&\frac{q^2(q^{-2(m-1)}+q^{2m}-q^2-1)}{(q^2-1)^2}
    \end{bmatrix}
\end{equation*}
The determinant of the matrix is:
\begin{equation*}
    \begin{split}
        &q^2\left(\frac{q^{-2(m-1)}+q^{2m}-q^2-1}{(q^2-1)^2}\right)^2-\frac{(q^2+1)^2(q^{-2}-q^{2(m-2)})(q^{-2(m-2)}-q^2)}{(1-q^2)^2}\\
        &=\frac{(q^{-2(m-1)}-1)^2(q^4-q^{2m})(1-q^{2(m+2)})}{q^2(q^2-1)^4}.
    \end{split}
\end{equation*}
Since $m\ge 3$, the determinant is always positive for $0< |q|< 1$. Thus, the matrix is invertible.

\subsubsection{Standard monomials with two segments of $aek$, $afh$, or $bdk$ }

\hfill \\

Consider equation basis $bdkbfg(ceg)^{m-2}$ with comparing basis $(aek)^{m-1}afh$. The derived linear relation includes $(bdk)^2(ceg)^{m-2}$, $bdkbfgcdh(ceg)^{m-3}$,\\ $bdkbfg(ceg)^{m-2}$, $(bfg)^2cdh(ceg)^{m-3}$, and $(bfg)^2(ceg)^{m-2}$. Thus, we can solve the Haar state of ${bdkbdk(ceg)^{m-2}}$ from this linear relation. 

\hfill

\noindent Similarly, consider equation basis $afhcdh(ceg)^{m-2}$ with comparing basis $(aek)^{m-2}bdk$. The derived linear relation includes $(afh)^2(ceg)^{m-2}$, $afhbfgcdh(ceg)^{m-3}$, \\$afhcdh(ceg)^{m-2}$, $bfg(cdh)^2(ceg)^{m-3}$, and $(cdh)^2(ceg)^{m-2}$. Thus, we can solve the Haar state of ${afhafh(ceg)^{m-2}}$ from this linear relation. 

\hfill

\noindent Next, by Equation~(\ref{apeq:7}) and Equation~(\ref{apeq:8}) in Appendix~\ref{apd:c}, we can compute the Haar state of ${afhbdk(ceg)^{m-2}}$ and ${bdkafh(ceg)^{m-2}}$. 

\hfill

\noindent Then, using the equality $bdk(ceg)^{m-2}=D_q*bdk(ceg)^{m-2}$, we can solve the Haar state of ${aekbdk(ceg)^{m-2}}$. Replacing $bdk(ceg)^{m-2}$ by $afh(ceg)^{m-2}$ in the above equation, we can find the Haar state of ${aekafh(ceg)^{m-2}}$. Finally, using the equality $aek(ceg)^{m-2}=aek*D_q*(ceg)^{m-2}$, we can solve the Haar state of ${(aek)^2(ceg)^{m-2}}$.

\hfill

\noindent At this point, we have computed all the Haar states of standard monomials ending with $(ceg)^{m-2}$ or ending with $bfgbfg(ceg)^{m-3}$, $cdhcdh(ceg)^{m-3}$, and $bfgcdh(ceg)^{m-3}$. Also, we computed the Haar state of standard monomials  $(bdk)^r(bfg)^{s}(cdh)^j(ceg)^{m-i-j}$ and $(afh)^r(cdh)^{s}(bfg)^j(ceg)^{m-i-j}$ with $r+s=1$. In the next subsection, we assume that the Haar state of standard monomials ending with $(ceg)^{m-i}$ or ending with $bfgbfg(ceg)^{m-i-1}$, $cdhcdh(ceg)^{m-i-1}$, and $bfgcdh(ceg)^{m-i-1}$ are known. We will also assume that the Haar state of standard monomials  $(bdk)^r(bfg)^{s}(cdh)^j(ceg)^{m-i-j}$ and $(afh)^r(cdh)^{s}(bfg)^j(ceg)^{m-i-j}$ with $r+s\le i-1$ are known. Based on this assumption, we will compute the Haar state of standard monomials ending with $(ceg)^{m-i-1}$ or ending with $bfgbfg(ceg)^{m-i-2}$, $cdhcdh(ceg)^{m-i-2}$, and $bfgcdh(ceg)^{m-i-2}$. We will also compute the Haar state of  standard monomials  $(bdk)^r(bfg)^{s}(cdh)^j(ceg)^{m-i-j}$ and $(afh)^r(cdh)^{s}(bfg)^j(ceg)^{m-i-j}$ with $r+s=i$.

\subsection{Standard monomials ending with $\boldsymbol{(ceg)^{m-i-1}}$ and standard monomials ending with$\boldsymbol{bfgbfg(ceg)^{m-i-2}}$, $\boldsymbol{cdhcdh(ceg)^{m-i-2}}$, and $\boldsymbol{bfgcdh(ceg)^{m-i-2}}$}

\hfill 

% We will apply an inductive approach to compute the pursuing Haar states. 
We can apply an inductive approach to compute the pursuing Haar states. 
% For a detailed discussion on each subsection, see Lu~\cite{note}.

\subsubsection{Monomials in form $(bdk)^r(bfg)^{s}(cdh)^j(ceg)^{m-i-j}$ and\\ $(afh)^r(cdh)^{s}(bfg)^j(ceg)^{m-i-j}$ with $r+s=i$ and $0\le j \le m-i$}

\hfill \\

First, notice that by our assumption, we already know the Haar states of \\ $(bdk)^r(bfg)^{s}(cdh)^j(ceg)^{m-i-j}$ and $(afh)^r(cdh)^{s}(bfg)^j(ceg)^{m-i-j}$ with $s\ge 1$ and $j=1$ since these monomials end with $bfgcdh(ceg)^{m-i-1}$. 

\hfill

\noindent To compute the Haar state of $(bdk)^icdh(ceg)^{m-i-1}$, we use equation basis $(bdk)^i(ceg)^{m-i}$ with comparing basis $(aek)^{m-1}afh$. Using the Theorem 1 e), we know that the derived linear relation only contains standard monomials in the form\\ $(bdk)^r(bfg)^s(cdh)^j(ceg)^{m-j-r-s}$ since no generator $a$ can appear in the newly generated monomials. Thus, the only monomial with unknown Haar state appearing in the linear relation is $(bdk)^icdh(ceg)^{m-i-1}$ and we can compute its Haar state. This finish the case $j=1$ for $(bdk)^r(bfg)^{s}(cdh)^j(ceg)^{m-i-j}$. Similarly, we can compute the Haar state of $(afh)^ibfg(ceg)^{m-i-1}$ using the linear relation derived from equation basis $(afh)^i(ceg)^{m-i}$ with comparing basis $(aek)^{m-1}bdk$.

\hfill

\noindent Now assume we know the Haar state of $(bdk)^r(bfg)^s(cdh)^j(ceg)^{m-i-j}$ for all $j\le t-1$. To compute the Haar state of case $j=t$, we use equation basis\\ $(bdk)^r(bfg)^s(cdh)^{t-1}(ceg)^{m-i-t+1}$ and comparing basis $(aek)^{m-1}afh$. Here, we have to compute the case $r=0$ first, then the case $r=1$, case $r=2$, until the case $r=i$. This is an analog to the zigzag recursive relation. To solve $(afh)^r(cdh)^{s}(bfg)^j(ceg)^{m-i-j}$, we use equation basis $(afh)^r(cdh)^{s}(bfg)^{j-1}(ceg)^{m-i-j+1}$ and comparing basis $(aek)^{m-1}bdk$ and use the same strategy as for monomials $(bdk)^r(bfg)^s(cdh)^j(ceg)^{m-i-j}$.

\subsubsection{Monomials with one high-complexity segment ending with $cdhcdh(ceg)^{m-i-2}$, $bfgbfg(ceg)^{m-i-2}$, and $bfgcdh(ceg)^{m-i-2}$}

\hfill \\

The monomials we are considering are in the form: $$aek(cdh)^{i+1}(ceg)^{m-i-2}, afh(cdh)^{i+1}(ceg)^{m-i-2}, bdk(cdh)^{i+1}(ceg)^{m-i-2};$$ and $$aek(bfg)^{i+1}(ceg)^{m-i-2}, afh(bfg)^{i+1}(ceg)^{m-i-2}, bdk(bfg)^{i+1}(ceg)^{m-i-2};$$ and
\begin{equation*}
    \begin{split}
        &aek(bfg)^{r+1}(cdh)^{s+1}(ceg)^{m-i-2}, afh(bfg)^{r+1}(cdh)^{s+1}(ceg)^{m-i-2}, \\
        &bdk(bfg)^{r+1}(cdh)^{s+1}(ceg)^{m-i-2}.
    \end{split}
\end{equation*}
We start with monomials ending with $cdhcdh(ceg)^{m-i-2}$. From subsection 4.2, we already know the Haar state of $bdk(cdh)^{i+1}(ceg)^{m-i-2}$ by the zigzag recursive relation. To compute the Haar state of $aek(cdh)^{i+1}(ceg)^{m-i-2}$ and $afh(cdh)^{i+1}(ceg)^{m-i-2}$, we build a $2\times 2$ linear system consisting of:
\begin{itemize}
    \item[1)] $D_q*(cdh)^{i+1}(ceg)^{m-i-2}=(cdh)^{i+1}(ceg)^{m-i-2}$;
    \item[2)] the linear relation derived from equation basis $afh(cdh)^{i}(ceg)^{m-i-1}$ and comparing basis $(aek)^{m-1}afh$.
\end{itemize}
Similarly, the Haar state of $afh(bfg)^{i+1}(ceg)^{m-i-2}$ is known from subsection 4.2. To compute the Haar state of $aek(bfg)^{i+1}(ceg)^{m-i-2}$ and $bdk(bfg)^{i+1}(ceg)^{m-i-2}$, we build a $2\times 2$ linear system consisting of:
\begin{itemize}
    \item[1)] $D_q*(bfg)^{i+1}(ceg)^{m-i-2}=(bfg)^{i+1}(ceg)^{m-i-2}$;
    \item[2)] the linear relation derived from equation basis $bdk(bfg)^{i}(ceg)^{m-i-1}$ and comparing basis $(aek)^{m-1}bdk$.
\end{itemize}

\noindent Then, we compute the Haar state of $aek(bfg)^{r+1}(cdh)^{s+1}(ceg)^{m-i-2}$, \\
$bdk(bfg)^{r+1}(cdh)^{s+1}(ceg)^{m-i-2}$, and $afh(bfg)^{r+1}(cdh)^{s+1}(ceg)^{m-i-2}$. By our assumption, we have solved the Haar state of $bdk(bfg)^{i-1-l}(cdh)^{2+l}(ceg)^{m-i-2}$,\\ $afh(cdh)^{i-1-l}(bfg)^{2+l}(ceg)^{m-i-2}$, $(bfg)^{i-l}(cdh)^{2+l}(ceg)^{m-i-2}$,\\ $(cdh)^{i-l}(bfg)^{2+l}(ceg)^{m-i-2}$, and $ceg(cdh)^{i-1-l}(bfg)^{2+l}(ceg)^{m-i-2}$ for $0\le l\le i-1$. Thus, to compute the Haar state of $aek(bfg)^{i-1-l}(cdh)^{2+l}(ceg)^{m-i-2}$, $0\le l\le i-1$, we use equation $D_q*(bfg)^{i-1-l}(cdh)^{2+l}(ceg)^{m-i-2}=(bfg)^{i-1-l}(cdh)^{2+l}(ceg)^{m-i-2}$.

\hfill

\noindent To compute the case $l=i$, notice that by our assumption the Haar states of $afh(bfg)^icdh(ceg)^{m-i-2}$ and $bdk(cdh)^ibfg(ceg)^{m-i-2}$ are known by our assumption. To construct a linear system of $aek(bfg)^icdh(ceg)^{m-i-2}$ and\\ $bdk(bfg)^icdh(ceg)^{m-i-2}$, we use linear relation derived from:
\begin{itemize}
    \item[1)] Equation $D_q*(bfg)^icdh(ceg)^{m-i-2}=(bfg)^icdh(ceg)^{m-i-2}$.
    \item[2)] Equation basis $bdk(bfg)^{i-1}cdh(ceg)^{m-i-1}$ and comparing basis $(aek)^{m-1}bdk$.
\end{itemize}
To construct a linear system of $aek(cdh)^ibfg(ceg)^{m-i-2}$ and $afh(cdh)^ibfg(ceg)^{m-i-2}$, we use linear relation derived from
\begin{itemize}
    \item[1)] Equation $D_q*(cdh)^ibfg(ceg)^{m-i-2}=(cdh)^ibfg(ceg)^{m-i-2}$.
    \item[2)] Equation basis $afh(cdh)^{i-1}bfg(ceg)^{m-i-1}$ and comparing basis $(aek)^{m-1}afh$.
\end{itemize}

\subsubsection{Monomials with two high-complexity segments ending with $(ceg)^{m-i-1}$}

\hfill\\

\noindent Finally, we compute the Haar states of monomials with two high-complexity segments ending with $(ceg)^{m-i-1}$. We start with monomials in form\\ $afhbdk(bfg)^r(cdh)^s(ceg)^{m-i-1}$ and $bdkafh(bfg)^r(cdh)^s(ceg)^{m-i-1}$. Notice that $afhbdkceg$ and $bdkafhceg$ can be written as a linear combination of $aekbfgcdh$ and other monomials with at most one high-complexity segment. Thus,\\ $afhbdk(bfg)^r(cdh)^s(ceg)^{m-i-1}$ and $bdkafh(bfg)^r(cdh)^s(ceg)^{m-i-1}$ can be written as a linear combination of $aek(bfg)^{r+1}(cdh)^{s+1}(ceg)^{m-i-2}$ and other monomials with at most one high-complexity segment. Thus, we can compute the Haar state of $afhbdk(bfg)^r(cdh)^s(ceg)^{m-i-1}$ and $bdkafh(bfg)^r(cdh)^s(ceg)^{m-i-1}$ using the Haar states we known. To compute the Haar state of $(bdk)^2(bfg)^r(cdh)^s(ceg)^{m-i-1}$, we use equation basis $bdk(bfg)^{r+1}(cdh)^s(ceg)^{m-i-1}$ and comparing basis $(aek)^{m-1}afh$. To compute the Haar state of $(afh)^2(bfg)^r(cdh)^s(ceg)^{m-i-1}$, we use equation basis $afh(bfg)^{r}(cdh)^{s+1}(ceg)^{m-i-1}$ and comparing basis $(aek)^{m-1}bdk$. At last, to compute the Haar state of $aekafh(bfg)^r(cdh)^s(ceg)^{m-i-1}$, we use the equation $D_q*afh(bfg)^r(cdh)^s(ceg)^{m-i-1}=afh(bfg)^r(cdh)^s(ceg)^{m-i-1}$. To compute the Haar state of $aekbdk(bfg)^r(cdh)^s(ceg)^{m-i-1}$, we use the equation $D_q*bdk(bfg)^r(cdh)^s(ceg)^{m-i-1}=bdk(bfg)^r(cdh)^s(ceg)^{m-i-1}$. To compute the Haar state of $aekaek(bfg)^r(cdh)^s(ceg)^{m-i-1}$, we use the equation\\ $aek*D_q*(bfg)^r(cdh)^s(ceg)^{m-i-1}=aek(bfg)^r(cdh)^s(ceg)^{m-i-1}$.

\hfill

\noindent At this point, we have solved the Haar state of all monomials with at most two high-complexity segments ending with $(ceg)^{m-i-1}$ and monomials with at most one high-complexity segment ending with $bfgbfg(ceg)^{m-i-2}$, $cdhcdh(ceg)^{m-i-2}$, and $bfgcdh(ceg)^{m-i-2}$.

\hfill

\noindent Starting from the next sub-section, we assume that the Haar states of monomials with at most $w\le i$ high-complexity segments ending with $(ceg)^{m-i-1}$ and monomials with at most $w-1$ high-complexity segments ending with $bfgbfg(ceg)^{m-i-2}$, $cdhcdh(ceg)^{m-i-2}$, and $bfgcdh(ceg)^{m-i-2}$ are known. Now, we compute the Haar states of monomials with $w+1$ high-complexity segments ending with $(ceg)^{m-i-1}$ and monomials with  $w$ high-complexity segments ending with $bfgbfg(ceg)^{m-i-2}$, $cdhcdh(ceg)^{m-i-2}$, and  $bfgcdh(ceg)^{m-i-2}$.

\subsubsection{Monomials with $w$ high-complexity segments ending with $bfgbfg(ceg)^{m-i-2}$, and $cdhcdh(ceg)^{m-i-2}$}

\hfill \\

We start with monomials ending with $cdhcdh(ceg)^{m-i-2}$. By subsection 4.4.2, we know the Haar state of standard monomials $(bdk)^r(bfg)^s(cdh)^{i-r-s+2}(ceg)^{m-i-2}$ with $r+s\le i$. Thus, we know the Haar state of $(bdk)^w(cdh)^{i-w+2}(ceg)^{m-i-2}$. Then, we construct a $2\times 2$ linear system of $aek(bdk)^{w-1}(cdh)^{i-w+2}(ceg)^{m-i-2}$ and $afh(bdk)^{w-1}(cdh)^{i-w+2}(ceg)^{m-i-2}$ consisting of:
\begin{itemize}
    \item[1)] $D_q*(bdk)^{w-1}(cdh)^{i-w+2}(ceg)^{m-i-2}=(bdk)^{w-1}(cdh)^{i-w+2}(ceg)^{m-i-2}$;
    \item[2)] linear relation derived from equation basis $afh(bdk)^{w-1}(cdh)^{i+1-w}(ceg)^{m-i-1}$ and comparing basis $(aek)^{m-1}afh$.
\end{itemize}
Next, we compute the Haar state of monomials in the form \\$aek(afh)^{j-1}(bdk)^{w-j}(cdh)^{i+2-w}(ceg)^{m-i-2}$ and $(afh)^j(bdk)^{w-j}(cdh)^{i+2-w}(ceg)^{m-i-2}$. We have solve the case $j=1$. Now, assume that we have solved all $j\le t-1$. When $j=1$, we have the following equation:
\begin{equation*}
    \begin{split}
        &(afh)^{t-1}(bdk)^{w-t}(cdh)^{i+2-w}(ceg)^{m-i-2}\\
        &=D_q*(afh)^{t-1}(bdk)^{w-t}(cdh)^{i+2-w}(ceg)^{m-i-2}\\
        &=aek(afh)^{t-1}(bdk)^{w-t}(cdh)^{i+2-w}(ceg)^{m-i-2}\\
        &-q*(afh)^t(bdk)^{w-t}(cdh)^{i+2-w}(ceg)^{m-i-2}\\
        &-q*bdk(afh)^{t-1}(bdk)^{w-t}(cdh)^{i+2-w}(ceg)^{m-i-2}\\
        &+q^2*bfg(afh)^{t-1}(bdk)^{w-t}(cdh)^{i+2-w}(ceg)^{m-i-2}\\
        &+q^2*cdh(afh)^{t-1}(bdk)^{w-t}(cdh)^{i+2-w}(ceg)^{m-i-2}\\
        &-q^3*ceg(afh)^{t-1}(bdk)^{w-t}(cdh)^{i+2-w}(ceg)^{m-i-2}
    \end{split}
\end{equation*}
By the Theorem 1 e), $bdk(afh)^{t-1}(bdk)^{w-t}(cdh)^{i+2-w}(ceg)^{m-i-2}$ can be written as a linear combination of $(afh)^{t-1}(bdk)^{w-t+1}(cdh)^{i+2-w}(ceg)^{m-i-2}$ and other standard monomials with at most $w-1$ high-complexity segments ending with $(ceg)^{m-i-2}$. Thus, the Haar state of $bdk(afh)^{t-1}(bdk)^{w-t}(cdh)^{i+2-w}(ceg)^{m-i-2}$ is known. Similarly, the Haar states of $bfg(afh)^{t-1}(bdk)^{w-t}(cdh)^{i+2-w}(ceg)^{m-i-2}$, $cdh(afh)^{t-1}(bdk)^{w-t}(cdh)^{i+2-w}(ceg)^{m-i-2}$, and\\ $ceg(afh)^{t-1}(bdk)^{w-t}(cdh)^{i+2-w}(ceg)^{m-i-2}$ are known. So the above equation is a linear relation between $aek(afh)^{t-1}(bdk)^{w-t}(cdh)^{i+2-w}(ceg)^{m-i-2}$ and \\$(afh)^t(bdk)^{w-t}(cdh)^{i+2-w}(ceg)^{m-i-2}$. For the other linear relation between the two monomials, we use the linear relation derived from equation basis \\$(afh)^t(bdk)^{w-t}(cdh)^{i+1-w}(ceg)^{m-i-1}$ and comparing basis $(aek)^{m-1}afh$. 

\hfill

\noindent Finally, we compute the Haar states of monomials in the form \\$(aek)^n(afh)^{j}(bdk)^{w-j-n}(cdh)^{i+2-w}(ceg)^{m-i-2}$. From the previous paragraph, we have solve the case $n=1$. Assume that we know the Haar state of all $n\le t-1$. to compute the case $n=t$, we use the quantum determinant condition 
\begin{equation*}
    \begin{split}
        &(aek)^{t-1}*D_q*(afh)^{j}(bdk)^{w-j-t}(cdh)^{i+2-w}(ceg)^{m-i-2}\\
        &=(aek)^{t-1}(afh)^{j}(bdk)^{w-j-t}(cdh)^{i+2-w}(ceg)^{m-i-2}, 
    \end{split}
\end{equation*}
in which the only monomial with unknown Haar state is \\$(aek)^t(afh)^{j}(bdk)^{w-j-t}(cdh)^{i+2-w}(ceg)^{m-i-2}$.

\hfill

\noindent The Haar state of monomials with $w$ high-complexity segments ending with\\ $bfgbfg(ceg)^{m-i-2}$ can be computed by an approach similar to the case of \\$cdhcdh(ceg)^{m-i-2}$ with every $afh$ segment replaced by $bdk$ and every $cdh$ segment replaced by $bfg$.

\subsubsection{Monomials in form $(afh)^w(cdh)^{j+1}(bfg)^{i-w-j+1}(ceg)^{m-i-2}$,\\ $bdk(afh)^{w-1}(cdh)^{j+1}(bfg)^{i-w-j+1}(ceg)^{m-i-2}$, and \\$aek(afh)^{w-1}(cdh)^{j+1}(bfg)^{i-w-j+1}(ceg)^{m-i-2}$ with $0\le j\le i-w$}

\hfill \\

From the previous sub-section, we know the Haar states of\\ $(afh)^w(cdh)^{j+1}(bfg)^{i-w-j+1}(ceg)^{m-i-2}$ for all $0\le j\le i-w-1$. Thus, when $0\le j\le i-w-1$, we only need to focus on $bdk(afh)^{w-1}(cdh)^j(bfg)^{i-w-j+2}(ceg)^{m-i-2}$ and $aek(afh)^{w-1}(cdh)^j(bfg)^{i-w-j+2}(ceg)^{m-i-2}$. To construct a linear system containing the two monomials, we use linear relation derived from 
\begin{itemize}
    \item[1)] $D_q*(afh)^{w-1}(cdh)^j(bfg)^{i-w-j+2}(ceg)^{m-i-2} \\=(afh)^{w-1}(cdh)^j(bfg)^{i-w-j+2}(ceg)^{m-i-2}$.
    \item[2)] Equation basis $(afh)^w(cdh)^{j}(bfg)^{i-w-j+1}(ceg)^{m-i-1}$ and comparing basis $(aek)^{m-1}bdk$.
\end{itemize}
When $j=i+1-w$, we need to solve the Haar state of $(afh)^w(cdh)^{i+1-w}bfg(ceg)^{m-i-2}$, $bdk(afh)^{w-1}(cdh)^{i+1-w}bfg(ceg)^{m-i-2}$, and $aek(afh)^{w-1}(cdh)^{i+1-w}bfg(ceg)^{m-i-2}$ at the same time. To construct a linear system containing the three monomials, we use linear relation derived from
\begin{itemize}
    \item[1)] $D_q*(afh)^{w-1}(cdh)^{i+1-w}bfg(ceg)^{m-i-2}=(afh)^{w-1}(cdh)^{i+1-w}bfg(ceg)^{m-i-2}$.
    \item[2)] Equation basis $(afh)^w(cdh)^{i-w}bfg(ceg)^{m-i-1}$ and comparing basis $(aek)^{m-1}afh$.
    \item[3)] Equation basis $bdk(afh)^{w-1}(cdh)^{i-w+1}(ceg)^{m-i-1}$ and comparing\\ basis $(aek)^{m-1}bdk$.
\end{itemize}

\subsubsection{Monomials in form $(aek)^n(bdk)^{j-n}(afh)^{w-j}(cdh)^{s+1}(bfg)^{r+1}(ceg)^{m-i-2}$ with $n \le j\le w$}

\hfill \\

To start, we compute the Haar states of monomials in form\\ $(bdk)^{j}(afh)^{w-j}(cdh)^{s+1}(bfg)^{r+1}(ceg)^{m-i-2}$ and \\
$aek(bdk)^{j-1}(afh)^{w-j}(cdh)^{s+1}(bfg)^{r+1}(ceg)^{m-i-2}$, with $1\le j\le w$ and $s+r=i-w$ at the same time. We already solve the case $j=1$. Without loss of generality, we assume that the Haar states of all $1\le j\le t-1\le w-1$ are known. To solve the case $j=t$, firstly, notice that we have the following equation: 
\begin{equation*}
    \begin{split}
        &(bdk)^{t-1}(afh)^{w-t}(cdh)^{s+1}(bfg)^{r+1}(ceg)^{m-i-2}  \\
        &=D_q*(bdk)^{t-1}(afh)^{w-t}(cdh)^{s+1}(bfg)^{r+1}(ceg)^{m-i-2} \\
        &=aek(bdk)^{t-1}(afh)^{w-t}(cdh)^{s+1}(bfg)^{r+1}(ceg)^{m-i-2}  \\
        &-q*(bdk)^{t}(afh)^{w-t}(cdh)^{s+1}(bfg)^{r+1}(ceg)^{m-i-2} \\
        &-q*afh(bdk)^{t-1}(afh)^{w-t}(cdh)^{s+1}(bfg)^{r+1}(ceg)^{m-i-2}    \\
        &+q^2*bfg(bdk)^{t-1}(afh)^{w-t}(cdh)^{s+1}(bfg)^{r+1}(ceg)^{m-i-2}  \\
        &+q^2*cdh(bdk)^{t-1}(afh)^{w-t}(cdh)^{s+1}(bfg)^{r+1}(ceg)^{m-i-2}  \\
        &-q^3*ceg(bdk)^{t-1}(afh)^{w-t}(cdh)^{s+1}(bfg)^{r+1}(ceg)^{m-i-2}.
    \end{split}
\end{equation*}
Here, $afh(bdk)^{t-1}(afh)^{w-t}(cdh)^{s+1}(bfg)^{r+1}(ceg)^{m-i-2}$ can be written as a linear combination of $(bdk)^{t-1}(afh)^{w-t+1}(cdh)^{s+1}(bfg)^{r+1}(ceg)^{m-i-2}$ and other monomials with less number of high-complexity segments. Thus, besides the Haar states of $aek(bdk)^{t-1}(afh)^{w-t}(cdh)^{s+1}(bfg)^{r+1}(ceg)^{m-i-2}$ and \\ $(bdk)^{t}(afh)^{w-t}(cdh)^{s+1}(bfg)^{r+1}(ceg)^{m-i-2}$, the Haar state of other monomials re known. So, in case $j=t$, we only need one more linear relation to construct a system of the two monomials. We will use the linear relation derived from equation basis $(bdk)^{j-1}(afh)^{w-j+1}(cdh)^s(bfg)^{r+1}(ceg)^{m-i-1}$ and comparing basis $(aek)^{m-1}afh$.

\hfill

\noindent Then, we compute the Haar state of monomials in the form \\$(aek)^n(bdk)^{j-n}(afh)^{w-j}(cdh)^{s+1}(bfg)^{r+1}(ceg)^{m-i-2}$ with $n \le j\le w$. We already compute the case of $n=1$. After solving the case of $n=t-1$, we can compute the case $n=t$ by the equation $(aek)^{t-1}(bdk)^{j-t}(afh)^{w-j}(cdh)^{s+1}(bfg)^{r+1}(ceg)^{m-i-2}=(aek)^{t-1}*D_q*(bdk)^{j-t}(afh)^{w-j}(cdh)^{s+1}(bfg)^{r+1}(ceg)^{m-i-2}$for all $t\le j\le w$.

\subsubsection{Monomials with $w+1$ high-complexity segments ending with $(ceg)^{m-i-1}$}

\hfill \\

We start with monomials in form $(aek)^n(bdk)^{j}(afh)^{w-j-n+1}(cdh)^s(bfg)^r(ceg)^{m-i-1}$ with $1\le j\le w$ and $0\le n\le w-j$. Monomials in this form contain at least one $afh$ segment and one $bdk$ segment. We can write the monomial as a linear combination of $(aek)^n(bdk)^{j-1}(afh)^{w-j-n}[afhbdk](cdh)^s(bfg)^r(ceg)^{m-i-1}$ and other monomials with less number of high-complexity segments. To compute the Haar state of $(aek)^n(bdk)^{j-1}(afh)^{w-j-n}[afhbdk](cdh)^s(bfg)^r(ceg)^{m-i-1}$, we can apply Equation~(\ref{apeq:7}) and Equation~(\ref{apeq:8}) in Appendix~\ref{apd:c} to rewrite it as a linear combination of monomials with known Haar states. 

\hfill

\noindent Next, we compute the Haar state of monomials in form $(afh)^{w+1}(cdh)^s(bfg)^r(ceg)^{m-i-1}$ and $(bdk)^{w+1}(cdh)^s(bfg)^r(ceg)^{m-i-1}$. To compute $(afh)^{w+1}(cdh)^s(bfg)^r(ceg)^{m-i-1}$, we use equation basis $(afh)^{w}(cdh)^{s+1}(bfg)^r(ceg)^{m-i-1}$ and comparing basis\\ $(aek)^{m-1}bdk$. To compute $(bdk)^{w+1}(cdh)^s(bfg)^r(ceg)^{m-i-1}$, we use equation basis $(bdk)^{w}(cdh)^s(bfg)^{r+1}(ceg)^{m-i-1}$ and comparing basis $(aek)^{m-1}afh$.

\hfill

\noindent Finally, we consider monomials in the form $(aek)^j(afh)^{w-j+1}(cdh)^s(bfg)^r(ceg)^{m-i-1}$ and $(aek)^j(bdk)^{w-j+1}(cdh)^s(bfg)^r(ceg)^{m-i-1}$ with $1\le j\le w$. To start, we compute the Haar state of $aek(afh)^{w}(cdh)^s(bfg)^r(ceg)^{m-i-1}$ using the equation $D_q*(afh)^{w}(cdh)^s(bfg)^r(ceg)^{m-i-1}=(afh)^{w}(cdh)^s(bfg)^r(ceg)^{m-i-1}$. If we have computed the Haar state for all $j\le n$, to compute the case $j=n+1$, we use the following equation $(aek)^{n}(afh)^{w-n}(cdh)^s(bfg)^r(ceg)^{m-i-1}=(aek)^{n}*D_q*(afh)^{w-n}(cdh)^s(bfg)^r(ceg)^{m-i-1}$. We can compute the Haar state of\\ $(aek)^j(bdk)^{w-j+1}(cdh)^s(bfg)^r(ceg)^{m-i-1}$ in the same way. When $j=w+1$, we use the equation $(aek)^{w}(cdh)^s(bfg)^r(ceg)^{m-i-1}=(aek)^{w}*D_q*(cdh)^s(bfg)^r(ceg)^{m-i-1}$ to compute the Haar state of $(aek)^{w+1}(cdh)^s(bfg)^r(ceg)^{m-i-1}$.

\hfill

\noindent At this point, we have solved the Haar state of all monomials with at most $w+1$ high-complexity segments ending with $(ceg)^{m-i-1}$ and monomials with at most $w$ high-complexity segments ending with $bfgbfg(ceg)^{m-i-2}$, $cdhcdh(ceg)^{m-i-2}$, and $bfgcdh(ceg)^{m-i-2}$. Using an induction argument, we can compute the Haar state of all monomials with at most $i+1$ high-complexity segments ending with $(ceg)^{m-i-1}$ and monomials with at most $i$ high-complexity segments ending with $bfgbfg(ceg)^{m-i-2}$, $cdhcdh(ceg)^{m-i-2}$, and $bfgcdh(ceg)^{m-i-2}$. Thus, this subsection shows that we can use induction on the value of $i$ from $1$ until $m-2$ and compute the Haar states of all monomials ending with $ceg$, $bfgbfg$, $cdhcdh$ or $bfgcdh$.  
\subsection{Monomials with at most one low-complexity segment}
\subsubsection{Monomials ending with $cdh$ or $bfg$}

\hfill \\

Here, we only show the procedure to compute monomials ending with $cdh$. The case of $bfg$ is solved similarly.

\hfill

\noindent We start with monomials in form $aek(bdk)^{i}(cdh)^{m-1-i}$ and $afh(bdk)^{i}(cdh)^{m-1-i}$ with $m-1-i\ge 1$. We have to solve the Haar states of the two monomials at the same time. The first linear relation comes from equation $D_q*(bdk)^{i}(cdh)^{m-i-1}=(bdk)^{i}(cdh)^{m-i-1}$.  The second linear relation is derived from equation basis \\$afh(bdk)^{i}(cdh)^{m-i-2}ceg$ and comparing basis $(aek)^{m-1}afh$.

\hfill

\noindent Next, we consider monomials in the form $aek(afh)^j(bdk)^{n-1}(cdh)^{m-n-j}$ and\\ $(afh)^{j+1}(bdk)^{n-1}(cdh)^{m-n-j}$ with $m-n-j\ge 1$, $j\ge 1$. We solve the case $j=0$ in the last paragraph. If we have known the Haar states of all $j\le t-1$, we can solve the case $j=t$ by linear relations derived from
\begin{itemize}
    \item[1)] $(afh)^{t}(bdk)^{n-1}(cdh)^{m-n-t}=D_q*(afh)^{t}(bdk)^{n-1}(cdh)^{m-n-t}$.
    \item[2)]  Equation basis $(afh)^{t+1}(bdk)^{n-1}(cdh)^{m-n-t-1}ceg$ and comparing\\ basis $(aek)^{m-1}afh$.
\end{itemize}
\noindent Finally, we compute the Haar states of monomials in the form $(aek)^n(afh)^j(bdk)^i(cdh)^{r}$ with $n\ge 2$, $r\ge 1$, and $i+j+n+r=m$. The case $n=1$ is solved in the last paragraph. Assume that we have solved the Haar state for all $j\le t-1$. To solve the case $n=t$, we use equation $(aek)^t(afh)^{j-1}(bdk)^n(cdh)^{r}=(aek)^t*D_q*(afh)^{j-1}(bdk)^n(cdh)^{r}$ where $j+n+r=m-t$ and $j,r\ge 1$.

\hfill

\noindent At this point, we have solved the Haar states of all monomials ending with $cdh$.

\subsubsection{Monomials without low-complexity segment}

\hfill \\

Now, we are able to solve the Haar states of all monomials with at least one low-complexity segment since the number of generators $c$ and $g$ cannot decrease. We start with monomial in form $(afh)^{m-i}(bdk)^i$. When $i\ge 2$, we use equation basis $(afh)^{m-i}(bdk)^{i-1}bfg$ and comparing basis $(aek)^{m-1}afh$. When $i=1$, we use equation basis $(afh)^{m-1}cdh$ and comparing basis $(aek)^{m-1}bdk$. Finally, we compute monomials in form $(aek)^n(afh)^r(bfg)^s$. To compute the case $n=1$, we use equation $D_q*(afh)^r(bfg)^s=(afh)^r(bfg)^s$ with $r+s=m-1$. Now, we assume that the Haar state of monomials in the form $(aek)^{n-1}(afh)^r(bdk)^s$ with $r+s=m-n+1$ are known. To compute $(aek)^{n}(afh)^{r'}(bdk)^{s'}$ with $r'+s'=m-n$, we use equation $(aek)^{n-1}(afh)^{r'}(bdk)^{s'}=(aek)^{n-1}*D_q*(afh)^{r'}(bdk)^{s'}$.

\hfill

\noindent At this point, we have computed the Haar states of all monomials of order $m$.

\section*{Acknowledgement}

\noindent
This author is advised by Professor Jeffrey Kuan from Texas A\&M University and Professor Micheal Brannan from University of Waterloo. This work is partially funded by National Science Foundation (NSF grant DMS-2000331).

\bibliographystyle{plain}
\bibliography{reference}

\appendix

% \input{tex/appendix_A}
% \newpage
% \input{tex/appendix_B}
\section{Useful equations} \label{apd:c}
When we switch the order of high complexity segments\\ with $ceg$ segment:
\begin{equation}
\begin{split}
    cegaek &= aekceg+(q^3 - q)*afhceg-(q - 1/q)*bdkceg-(q^2 - 1)^2/q*bfgcdh,\\
    cegafh &= q^2*afhceg+(1 - q^2)*bfgcdh,\\
    cegbdk &= q^{-2}*bdkceg+(1 - q^{-2})*bfgcdh,
    \label{apeq:1}
\end{split}
\end{equation}
with $cdh$ segment:
\begin{equation}
    \begin{split}
        cdhaek &= aekcdh+(q^4 - q^2)*afhceg+(1 - q^2)*bdkceg-(q^2 - 1)^2*bfgcdh,\\
        cdhafh &= afhcdh+(q^3 - q)*afhceg-(q^3 - q)*bfgcdh,\\
        cdhbdk &= bdkcdh-(q - 1/q)*bdkceg+(q - 1/q)*bfgcdh,
        \label{apeq:2}
    \end{split}
\end{equation}
with $bfg$ segment:
\begin{equation}
    \begin{split}
        bfgaek &= aekbfg+(q^4 - q^2)*afhceg+(1 - q^2)*bdkceg-(q^2 - 1)^2*bfgcdh,\\
        bfgafh &= afhbfg+(q^3 - q)*afhceg-(q^3 - q)*bfgcdh,\\
        bfgbdk &= bdkbfg-(q - 1/q)*bdkceg+(q - 1/q)*bfgcdh.
        \label{apeq:3}
    \end{split}
\end{equation}
The key observation is that when we switch the order of a high-complexity segment with a low-complexity segment, the newly generated monomials contain at most one high-complexity segment.

\hfill

\noindent When we switch the order of two high complexity segments:
\begin{equation}
    \begin{split}
        bdkafh =& q^{-2}*afhbdk+(1 - q^{-2})*aekbfg\\
        &+(1 - q^{-2})*aekcdh-(q^2 - 1)^2/q^3*aekceg\\
        &+\frac{(q^2 - 1)^2(q^2 + 1)}{q^2}*afhceg-(q^4 - q^2)*bfgcdh,
        \label{apeq:4}
    \end{split}
\end{equation}
\begin{equation}
    \begin{split}
        afhaek =& aekafh+(q - 1/q)*afhbdk-(q - 1/q)*aekbfg\\
        &-(q - 1/q)*aekcdh+(q - 1/q)^2*aekceg+(q - 1/q)*afhceg,\\
        \label{apeq:5}
    \end{split}
\end{equation}
\begin{equation}
    \begin{split}
        bdkaek =& aekbdk-(q - 1/q)*afhbdk+(q - 1/q)*aekbfg\\
        &+(q - 1/q)*aekcdh-(q - 1/q)^2*aekceg\\
        &+\frac{(q^2 - 1)^2(q^2 + 1)}{q}*afhceg-(q^3 - q)*bdkceg\\
        &-q(q^2 - 1)^2*bfgcdh.
        \label{apeq:6}
    \end{split}
\end{equation}
In Equation~(\ref{apeq:4}), the newly generated monomials contain at most one high-complexity segment. In Equation~(\ref{apeq:5}) and Equation~(\ref{apeq:6}), the newly generated monomials contain at most one high-complexity segment except the monomial $afhbdk$.

\hfill

\noindent Standard monomials $afhbdkceg$, $bdkafhceg$ and $aekbfgcdh$ have the same counting matrix:
\begin{equation*}
    \begin{bmatrix}
    1&1&1\\
    1&1&1\\
    1&1&1
    \end{bmatrix}.
\end{equation*}
We have the following equation:
\begin{equation}
    \begin{split}
        afhbdkceg =& q*aekbfgcdh+(1 - q^2)*aekbfgceg\\
        &+(1 - q^2)*aekcdhceg+(q^2 - 1)^2/q*aek(ceg)^2\\
        &+(1 - q^2)*afhbfgcdh+(q^3 - q)*afhbfgceg\\
        &+(q^3 - q)*afhcdhceg-(q^2 - 1)^2*afh(ceg)^2.
        \label{apeq:7}
    \end{split}
\end{equation}
\begin{equation}
    \begin{split}
        bdkafhceg =& 1/q*aekbfgcdh-(1 - q^{-2})*afhbfgcdh\\
        &+(q - q^{-1})*afhbfgceg+(q - q^{-1})*afhcdhceg\\
        &+(q^2 - 1)^2*afh(ceg)^2-(q^4 - q^2)*bfgcdh(ceg)^2.
        \label{apeq:8}
    \end{split}
\end{equation}
\newpage
\section{Illustration of the Zigzag Recursive Relation} \label{apd:d}

\tikzstyle{smallbox} = [rectangle,  minimum width=0.5cm, minimum height=0.5cm, text centered, draw=white, fill=white!100]
\tikzstyle{box} = [rectangle, minimum width=3.7cm, minimum height=1cm, text centered, draw=black, fill=white!100]
\tikzstyle{arrow} = [thick,->,>=stealth]
\tikzstyle{arrow_v} = [thick,->,>=stealth, double distance=1pt]

\begin{figure}[htb!]
\centering
\resizebox{0.75\textwidth}{!}{

\begin{tikzpicture}
\node (label1) [smallbox] {\rotatebox{180}{$\mathbf{\vdots}$}};
\node (label2) [smallbox, right of=label1, xshift= 5cm]   {\rotatebox{180}{$\mathbf{\vdots}$}};
\node (label3) [box, below of=label1, yshift= -0.8cm]      {$bfg(cdh)^{j-1}(ceg)^{n-j+1}$};
\node (label4) [box, right of=label3, xshift= 5cm]      {$bdk(cdh)^{j-1}(ceg)^{n-j+1}$};
\node (label5) [box, below of=label3, yshift= -0.8cm]      {$bfg(cdh)^{j}(ceg)^{n-j}$};
\node (label6) [box, right of=label5, xshift= 5cm]      {$bdk(cdh)^{j}(ceg)^{n-j}$};
\node (label7) [box, below of=label5, yshift= -0.8cm]      {$bfg(cdh)^{j+1}(ceg)^{n-j-1}$};
\node (label8) [box, right of=label7, xshift= 5cm]      {$bdk(cdh)^{j+1}(ceg)^{n-j-1}$};
\node (label9) [smallbox, below of=label7, yshift= -0.8cm]      {\rotatebox{0}{$\mathbf{\vdots}$}};
\node (label10) [smallbox, right of=label9, xshift= 5cm]      {\rotatebox{0}{$\mathbf{\vdots}$}};

\draw [arrow_v] (label1) -- node[anchor=east] {1} (label3);
\draw [arrow_v] (label3) -- node[anchor=east] {3} (label5);
\draw [arrow_v] (label5) -- node[anchor=east] {5} (label7);
\draw [arrow_v] (label7) -- node[anchor=east] {7} (label9);

\draw [arrow_v] (label2) -- node[anchor=west] {2} (label4);
\draw [arrow_v] (label4) -- node[anchor=west] {4} (label6);
\draw [arrow_v] (label6) -- node[anchor=west] {6} (label8);
\draw [arrow_v] (label8) -- node[anchor=west] {8} (label10);

\draw [arrow] (label3) -- node[anchor=south] {} (label4);
\draw [arrow] (label5) -- node[anchor=south] {} (label6);
\draw [arrow] (label7) -- node[anchor=south] {} (label8);

\draw [arrow] (label4) -- node[anchor=south, xshift=-0.1cm, yshift=-0.1cm] {\rotatebox{30}{}} (label5);
\draw [arrow] (label6) -- node[anchor=south, xshift=-0.1cm, yshift=-0.1cm] {\rotatebox{30}{}} (label7);

\end{tikzpicture}
}
\caption{Illustration of the Zigzag recursive relation. Arrow 1 to 8 means that we are using the monomial in the upper box as the equation basis and $(aek)^{m-1}afh$ as the comparing basis to derive the linear relation containing the monomial in the lower box as the only unknown. The zigzagging arrows indicate the order that we follow to solve the Haar states of these monomials. }
\label{fig:my_label}
\end{figure}
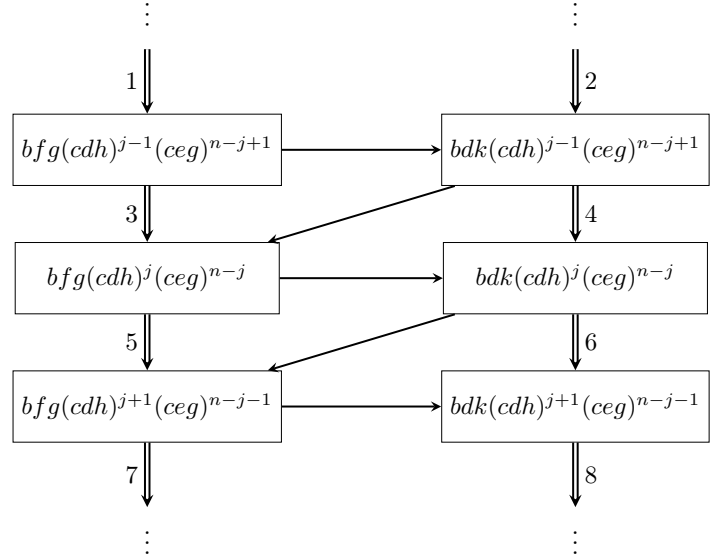
\newpage
\section{Example of $q$-Deformed Weingarten Function} \label{apd:e}
We know that when $q\rightarrow 1$, $\mathcal{O}(SU_q(n))$ becomes $SU(n)$ and the Haar state on $SU_q(n)$ becomes the Haar measure on $SU(n)$. This implies that 
\begin{equation*}
\begin{split}
    {h(x_{i_1j_1}\cdots x_{i_nj_n}x_{i_1'j_1'}^*\cdots x_{i_n'j_n'}^*)
    \xrightarrow{q\rightarrow 1}
    \int_{SU(n)}U_{i_1j_1}\cdots U_{i_nj_n}U_{i_1'j_1'}^*\cdots U_{i_n'j_n'}^*\ dU},
\end{split}
\end{equation*}
where $x_{i,j}$'s are generators of $\mathcal{O}(SU_q(n))$ and $U_{i,j}$'s are coordinate function on $SU(n)$. The Haar state on the quantum sphere serves as an example of $q$-deformed Weingarten function on $SU(n)$(for detail, see Noumi \textit{et al.}~\cite{noumi1993finite}, Reshetikhin \textit{et al.}~\cite{reshetikhin2001quantum}, Mikkelsen \textit{et al.}~\cite{mikkelsen2022haar}).

\hfill

\noindent One major difference between the Haar state and the integral is that the order of generators affects the Haar state. However, the order of the coordinate functions does not affect the integral. In the following examples on $\mathcal{O}(SU_q(3))$, we show that the order of generators in the Haar state does not affect the limit at $q=1$.

\hfill

\noindent \textbf{Example 1:}
\begin{equation*}
\begin{split}
h(x_{11}x_{22}x_{11}^*x_{22}^*)&=h(aea^*e^*)=h(ae(ek-q\cdot fh)(ak-q\cdot cg))\\
    &=h(aeekak)-q\cdot h(aefhak)-q\cdot h(aeekcg)+q^2\cdot h(aefhcg)\\
    &\frac{q^2}{(q^2+1)^2(q^4+1)}.
\end{split}
\end{equation*}
\begin{equation*}
    \begin{split}
h(x_{22}x_{11}x_{11}^*x_{22}^*)&=h(eaa^*e^*)=h(ea(ek-q\cdot fh)(ak-q\cdot cg))\\ 
    &=h(eaekak)-q\cdot h(eafhak)-q\cdot h(eaekcg)+q^2\cdot h(eafhcg)\\
    &=\frac{q^2}{(q^2+1)^2(q^4+1)}\frac{q^6+q^2+1}{q^4+q^2+1}
    \end{split}.
\end{equation*}
\begin{equation*}
    \begin{split}
        h(x_{11}x_{22}x_{22}^*x_{11}^*)&=h(aee^*a^*)=h(ae(ak-q\cdot cg)(ek-q\cdot fh))\\
        &=h(aeakek)-q\cdot h(aecgek)-q\cdot h(aeakfh)+q^2\cdot h(aecgfh)\\
        &=\frac{1}{(q^2+1)^2(q^4+1)}\frac{q^6+q^4+1}{q^4+q^2+1}.
    \end{split}
\end{equation*}
\begin{equation*}
    \begin{split}
        h(x_{11}x_{11}^*x_{22}x_{22}^*)&=h(aa^*ee^*)=h((aek-q\cdot afh)(eak-q\cdot ceg))\\
        &=h(aekeak)-q\cdot h(aekceg)-q\cdot h(afheak)+q^2\cdot h(afhceg)\\
       &=\frac{q^2}{(q^2+1)^2(q^4+1)}.
    \end{split}
\end{equation*}
The Haar states of monomials in other orders can be computed by the relation $h(y\phi(x))=h(xy)$ where $\phi$ is the homomorphism on $\mathcal{O}(SU_q(3))$ such that $\phi(x_{ij})=q^{2(i+j-4)}x_{ij}$. When $q\rightarrow 1$, all Haar state values goes to $1/8$ which is consistent with 
\begin{equation*}
    \int_{SU(3)}U_{11}U_{22}U_{11}^*U_{22}^*\ dU=Wg(1^2,3)=\frac{1}{3^2-1}=\frac{1}{8}.
\end{equation*}
\textbf{Example 2:}
\begin{equation*}
    \begin{split}
        h(x_{11}x_{32}x_{31}^*x_{12}^*)&=h(ahg^*b^*)=(-q)^{-1}\cdot h(ah(bf-q\cdot ce)(dk-q\cdot fg))\\
        &=-q^{-1}[h(ahbfdk)-q\cdot h(ahcedk)-q\cdot h(ahbffg)+q^2\cdot h(ahcefg)]\\
        &=\frac{-q}{(q^2+1)^2(q^4+1)(q^4+q^2+1)}.
    \end{split}
\end{equation*}
\begin{equation*}
    \begin{split}
        h(x_{32}x_{11}x_{31}^*x_{12}^*)&=h(hag^*b^*)=(-q)^{-1}\cdot h(ha(bf-q\cdot ce)(dk-q\cdot fg))\\
        &=-q^{-1}[h(habfdk)-q\cdot h(hacedk)-q\cdot h(habffg)+q^2\cdot h(hacefg)]\\
        &=\frac{-q^7}{(q^2+1)^2(q^4+1)(q^4+q^2+1)}.
    \end{split}
\end{equation*}
\begin{equation*}
    \begin{split}
        h(x_{11}x_{32}x_{12}^*x_{31}^*)&=h(ahb^*g^*)=(-q)^{-1}\cdot h(ah(dk-q\cdot fg)(bf-q\cdot ce))\\
        &=-q^{-1}[h(ahdkbf)-q\cdot h(ahdkce)-q\cdot h(ahfgbf)+q^2\cdot h(ahfgce)]\\
        &=\frac{-q}{(q^2+1)^2(q^4+1)(q^4+q^2+1)}.
    \end{split}
\end{equation*}
\begin{equation*}
    \begin{split}
        h(x_{11}x_{12}^*x_{32}x_{31}^*)&=h(ab^*hg^*)=(-q)^{-1}\cdot h(a(dk-q\cdot fg)h(bf-q\cdot ce))\\
        &=-q^{-1}[h(adkhbf)-q\cdot h(adkhce)-q\cdot h(afghbf)+q^2\cdot h(afghce)]\\
        &=\frac{-q^4}{(q^2+1)^2(q^4+1)(q^4+q^2+1)}.
    \end{split}
\end{equation*}
When $q\rightarrow 1$, all Haar state values goes to $-1/24$ which is consistent with 
\begin{equation*}
    \int_{SU(3)}U_{11}U_{32}U_{31}^*U_{12}^*\ dU=Wg(2,3)=-\frac{1}{3(3^2-1)}=-\frac{1}{24}.
\end{equation*}
\textbf{Example 3:}
\begin{equation*}
    \begin{split}
        h(x_{11}x_{11}x_{11}^*x_{11}^*)&=h(aaa^*a^*)=h(aa(ek-q\cdot fh)(ek-q\cdot fh))\\
        &=h(aaekek)-q\cdot h(aafhek)-q\cdot h(aaekfh)+q^2\cdot h(aafhfh)\\
        &=\frac{1}{(q^4+1)(q^4+q^2+1)}.
    \end{split}
\end{equation*}
\begin{equation*}
    \begin{split}
        h(x_{11}x_{11}^*x_{11}x_{11}^*)&=h(aa^*aa^*)=h(a(ek-q\cdot fh)a(ek-q\cdot fh))\\
        &=h(aekaek)-q\cdot h(afhaek)-q\cdot h(aekafh)+q^2\cdot h(afhafh)\\
        &=\frac{q^4-q^2+1}{(q^4+1)(q^4+q^2+1)}.
    \end{split}
\end{equation*}
When $q\rightarrow 1$, all Haar state values goes to $1/6$ which is consistent with 
\begin{equation*}
\begin{split}
    \int_{SU(3)}U_{11}^2(U_{11}^*)^2\ dU&=2Wg(1^2,3)+2Wg(2,3)\\
    &=2\frac{1}{3^2-1}-2\frac{1}{3(3^2-1)}=\frac{1}{6}.
\end{split}  
\end{equation*}

\end{document}